\documentclass[12pt,reqno]{amsart}
\usepackage[margin=1in]{geometry}
\usepackage[utf8]{inputenc}
\usepackage{setspace}
\usepackage{hyperref}
\usepackage{enumitem}
\usepackage{graphicx}
\usepackage{xcolor}
\usepackage[bitstream-charter]{mathdesign} 
\usepackage{amsmath, mathtools}
\usepackage{amsthm}
\usepackage{enumitem}
\usepackage{caption}
\usepackage{subcaption}
\usepackage{algorithm2e}

\usepackage{todonotes}
\usepackage{tikz}
\usetikzlibrary{shapes}
\usepackage[normalem]{ulem}

\newtheorem{theorem}{Theorem}
\newtheorem{prop}{Proposition}
\newtheorem{observation}{Observation}

\newtheorem{definition}{Definition}
\newtheorem{conjecture}{Conjecture}
\newtheorem{lemma}{Lemma}

\newtheorem{claim}{Claim}

\hypersetup{
    colorlinks = false,
    allbordercolors = {white},
}

\tikzstyle{vert}=[shape=circle,draw=black,fill=black, inner sep=.75mm]
\tikzstyle{fixed}=[shape=rectangle,draw=black,fill=white, inner sep=2mm]
\tikzstyle{fixed2}=[shape=circle,draw=black,fill=white, inner sep=1.2mm]
\tikzstyle{uncolored}=[dashed,thick]
\tikzstyle{uncolored2}=[dotted,thick]
\tikzstyle{red?}=[dashed,thick,color=red]
\tikzstyle{blue?}=[dashed,thick,color=blue]
\tikzstyle{purple?}=[dashed,thick,color=purple]
\tikzstyle{purple}=[solid,thick,color=purple]
\tikzstyle{red}=[solid,thick,color=red]
\tikzstyle{blue}=[solid,thick,color=blue]
\tikzstyle{green}=[solid,thick,color=green]
\tikzstyle{green?}=[dashed,thick,color=green]

\newcommand{\DrawGamma}[4]{

	\node (label) at (0,.5) [] {#4};
	
    \node (g1) at (-1.37,.5) [vert] {};
	\node (g2) at (-1.37,-.5) [vert] {};
	\node (g3) at (-.5,0) [vert] {};
	\node (g4) at (.5,0) [vert] {};
	\node (g5) at (1.37,.5) [vert] {};
	\node (g6) at (1.37,-.5) [vert] {};
	
	\draw [thick,#1] (g1) -- node[left] {} (g2) ; 
	\draw [thick,#2] (g1) -- node[above] {} (g3) ; 
	
	\draw [thick,#2] (g2) -- node[below] {} (g3) ; 
	\draw [thick,#1] (g3) -- node[above] {} (g4) ; 
	
	\draw [thick,#3] (g4) -- node[above] {} (g5) ; 
	\draw [thick,#3] (g4) -- node[below] {} (g6) ; 
	
	\draw [thick,#1] (g5) -- node[right] {} (g6) ; 

}

\newcommand{\DrawGammaSpecial}[4]{
	
	\node (label) at (0,.5) [] {#4};
	
    \node (g1) at (-1.37,.5) [vert] {};
	\node (g2) at (-1.37,-.5) [vert] {};
	\node (g3) at (-.5,0) [vert] {};
	\node (g4) at (.5,0) [vert] {};
	\node (g5) at (1.37,.5) [vert] {};
	\node (g6) at (1.37,-.5) [vert] {};
	
	\draw [thick, #1] (g1) -- node[left] {} (g2) ; 
	\draw [thick, #2] (g1) -- node[above] {} (g3) ; 
	
	\draw [thick, #2] (g2) -- node[below] {} (g3) ; 
	\draw [thick, #1] (g3) -- node[above] {} (g4) ; 
	
	\draw [thick, #3] (g4) -- node[above] {} (g5) ; 
	\draw [thick, #3] (g4) -- node[below] {} (g6) ; 
	
	\draw [#1] (g5) -- node[right] {} (g6) ; 
	
	\foreach \x/\y in {g1/g5, g1/g6, g1/g4, g2/g6, g2/g4, g2/g5, g5/g3, g6/g3}
		\draw [dashed] (\x) -- node[right] {} (\y) ; 

}

\newcommand{\E}{E}

\renewcommand{\d}{\delta}

\newcommand{\rbrac}[1]{\left(#1\right)} 

\newcommand{\mc}[1]{\mathcal{#1}}

\newcommand{\eps}{\varepsilon}

\newcommand{\nn}{\nonumber}

\def\G{\Gamma}
\def\g{\gamma}

\allowdisplaybreaks

\title{The generalized Ramsey number $f(n, 5, 8) = \frac 67 n + o(n)$}

\author{Patrick Bennett}
\address{Department of Mathematics, Western Michigan University, Kalamazoo, MI, USA}
\thanks{The first author was supported in part by Simons Foundation Grant \#426894.}
\email{\tt patrick.bennett@wmich.edu}

\author{Ryan Cushman}\thanks{}
\address{Department of Mathematics,
University of Wisconsin-Eau Claire, Eau Claire, WI, USA} 
\email{\tt cushmarj@uwec.edu}

\author{Andrzej Dudek}
\address{Department of Mathematics, Western Michigan University, Kalamazoo, MI, USA}
\thanks{The third author was supported in part by Simons Foundation Grant \#522400.}
\email{\tt andrzej.dudek@wmich.edu}

\begin{document}

\maketitle 

\begin{abstract}
    A \emph{$(p, q)$-coloring} of $K_n$ is a coloring of the edges of $K_n$ such that every $p$-clique has at least $q$ distinct colors among its edges. The generalized Ramsey number $f(n, p, q)$ is the minimum number of colors such that $K_n$ has a $(p, q)$-coloring. Gomez-Leos, Heath, Parker, Schweider and Zerbib recently proved $f(n, 5, 8) \ge \frac 67 (n-1)$. Here we prove an asymptotically matching upper bound. 
\end{abstract}

\section{Introduction}

In the 1970's, Erd\H{o}s and Shelah \cite{E75} proposed the following generalization of the classical Ramsey problem. 
 \begin{definition}
Fix  integers $p, q$ such that $p \ge 3$ and $2 \le q \le \binom p2$. A \emph{$(p, q)$-coloring} of $K_n$ is a coloring of the edges of $K_n$ such that every $p$-clique has at least $q$ distinct colors among its edges. The generalized Ramsey number $f(n, p, q)$ is the minimum number of colors such that $K_n$ has a $(p, q)$-coloring.
\end{definition}
Note that a $(p, 2)$-coloring only avoids monochromatic $p$-cliques, so the problem reduces to the standard multicolor Ramsey problem. In general, the problem is to estimate $f(n, p, q)$ for fixed $p, q$ and $n$ large. Erd\H{o}s and Gy\'arf\'as \cite{EG97} systematically studied this problem and obtained many bounds for various ranges of $p$ and $q$. For our current paper, the most relevant result of Erd\H{o}s and Gy\'arf\'as \cite{EG97} is the following. For arbitrary $p$ and 
\begin{equation}\label{eqn:qlin}
   q=q_{\text{lin}}(p):=\binom p2-p+3, 
\end{equation}
$f(n,p,q)$ is linear in $n$ (proving linear lower and upper bounds with different coefficients), but that $f(n,p,q-1)$ is sublinear. Note that $q_{\text{lin}}(4)=5$. Erd\H{o}s and Gy\'arf\'as \cite{EG97} paid special attention to the case of $f(n, 4, 5)$ and proved more specific bounds for this case: 
\[\frac 56 n + o(n) \le f(n, 4, 5) \le n+ o(n).\] 
 Erd\H{o}s and Gy\'arf\'as had differing opinions on which bound should be closer to the truth. Erd\H{o}s believed that the upper bound was essentially correct, while Gy\'arf\'as thought it should be closer to the lower bound. Together with Pra\l at, the current authors \cite{BCDP22} proved that $ f(n, 4, 5)=\frac 56 n + o(n)$ and so Gy\'arf\'as was right (this result was later reproved by Joos and Mubayi \cite{JM22}, which we briefly discuss below). The coloring we used in \cite{BCDP22} was obtained by a randomized coloring process which we analyzed using the differential equation method pioneered and popularized by Wormald (for a gentle introduction by the first and third authors, see \cite{BD22}).
 
 Recently two teams of researchers developed powerful ``black box'' theorems which can be used to analyze a certain family of random processes. With a very minor technical modification, these theorems can be applied to the process we analyzed in~\cite{BCDP22}. Indeed, Glock, Joos, Kim, K\"uhn and Lichev \cite{GJKKL} used the differential equation method  to greedily build a matching in a hypergraph such that the matching is forbidden from containing certain sets of edges they call {\em conflicts}. Delcourt and Postle \cite{DP22} used a semi-random ``nibble'' method to obtain similar results. Joos and Mubayi \cite{JM22} showed that the process we analyzed in \cite{BCDP22} could be encoded as an instance of this more general process analyzed in \cite{GJKKL}, as well as obtaining a few results on related generalized Ramsey problems by using similar encodings. 

 In this paper, we asymptotically determine $f(n, 5, q_{\text{lin}}(5))$. Notice that by \eqref{eqn:qlin}, $q_{\text{lin}}(5)=8$.
\begin{theorem}\label{thm:main}
We have
    \[
    f(n, 5, 8) = \frac 67 n + o(n). 
    \]
\end{theorem}
Gomez-Leos, Heath, Parker, Schweider and Zerbib \cite{GHPSZ23} recently proved that $\frac 67 (n-1) \le f(n, 5, 8) \le n + o(n)$ (the upper bound was also independently proved in~\cite{BCD23}). In this paper we give an alternate proof for this lower bound, since we believe it illuminates what must be done to get a matching upper bound. But our main contribution is to obtain a $(5, 8)$-coloring using $\frac 67 n + o(n)$ colors. To do this we will actually separately apply certain results of Delcourt and Postle \cite{DP22} as well as Glock, Joos, Kim, K\"uhn and Lichev \cite{GJKKL}. The results in \cite{DP22} and \cite{GJKKL} are quite similar in the big picture, but there are important technical differences that give each one an advantage over the other in certain instances. 

\section{Gadget packing and conflict-free hypergraph matchings}

In this section we describe some extremal problems which researchers have made significant progress on in recent years, especially since the results of Delcourt and Postle \cite{DP22} and Glock, Joos, Kim, K\"uhn and Lichev \cite{GJKKL}. The recent progress on many of these problems used ideas and methods that serve as a good warm-up for how we will obtain our coloring. 

The following problem was first studied by Brown, Erd\H{o}s and S\'os~\cite{BES1973}. 
\begin{definition}
 Let $\mc{H}$ be a $3$-uniform hypergraph. An \emph{$(s, k)$-configuration} in $\mc{H}$ is a set of $s$ vertices inducing $k$ or more edges. We say $\mc{H}$ is \emph{$(s, k)$-free} if it has no $(s, k)$-configuration. Let $F(n; s, k)$ be the largest possible number of edges in an $(s, k)$-free $3$-uniform hypergraph with $n$ vertices. In terms of classical extremal numbers, 
\[
  F(n; s, k) =  \mbox{ex}_{r}(n, \mc{G}_{s, k}),
\]
 where $\mc{G}_{s, k}$ is the family of all $3$-uniform hypergraphs on $s$ vertices and $k$ edges.
\end{definition}
Brown, Erd\H{o}s and S\'os~\cite{BES1973} made two well-known conjectures on this problem, and the relevant one for us is as follows.
\begin{conjecture}\label{conj:BES}
    For all $\ell \ge 4$ the following limit exists:
\[
    \lim_{n \rightarrow \infty} \frac{F\rbrac{n; \ell, \ell-2}}{n^2}.
\]
\end{conjecture}
For $\ell=4$ above we get $F(n; 4, 2)$, which is just the maximum number of edges in a linear $3$-uniform hypergraph, i.e.\ a partial Steiner system. Since almost-perfect partial Steiner systems exist, we have $F(n; 4, 2) = \frac 16 n^2 + o(n^2)$. 

For the case $\ell=5$, Glock \cite{G19} proved Conjecture \ref{conj:BES} by showing that $F(n; 5, 3) = \frac 15 n^2 +o(n)$. We briefly sketch a proof here, since it is a good warm-up. First we sketch the upper bound. Suppose $\mc{H}$ is $(5, 3)$-free. Let $x_1$ be the number of edges that share at most one vertex with every other edge. Let $x_2$ be the number of pairs of edges that share two vertices. Since $\mc{H}$ is $(5, 3)$-free, 
\begin{equation}\label{eqn:glock1}
    |\mc{H}| = x_1 + 2x_2 .
\end{equation}
Now we will obtain an inequality restricting $x_1, x_2$ as follows. For any edge $e = xyz$ counted by $x_1$ we say $e$ {\em hits} all three pairs $xy, yz, xz$ contained in $e$. For any pair of edges $e,e'$ counted by $x_2$ we say the pair $e, e'$ hits all 5 pairs of vertices contained in $e$ or $e'$. Since each pair of vertices is hit at most once, we conclude that 
\begin{equation}\label{eqn:glock2}
    3x_1 + 5x_2 \le \binom n2.
\end{equation}
Now, \eqref{eqn:glock1} and \eqref{eqn:glock2} give
\[
|\mc{H}| = x_1 + 2x_2 = \frac{2}{5}\left(\frac{5}{2}x_1 + 5x_2 \right)
\le \frac{2}{5}\left(3x_1 + 5x_2 \right)
\le \frac{2}{5} \binom n2
\le \frac 15 n^2.
\]
In the next paragraph we sketch a matching lower bound (up to some $o(n)$ error term). 

One appealing aspect of the proof for the upper bound  is that it illuminates how to obtain a construction that asymptotically matches it. First, to asymptotically maximize \eqref{eqn:glock1}  subject to the constraint \eqref{eqn:glock2} we must have $x_1 = o(n^2), x_2 = \frac{1}{10}n^2 + o(n^2)$. So our hypergraph consists almost entirely of pairs of edges sharing two vertices. In other words, we would like to pack copies of the hypergraph $\mc{G}$ shown in Figure~\ref{fig:glock}. Our constraint \eqref{eqn:glock2} must be asymptotically tight, so we must hit nearly every pair of vertices. Recall that each pair $e=xyz,\ e'=wyz$ counted by $x_2$ hits $xy, yz, xz, wy, wz$ but not the pair $xw$. Since almost every pair must be hit, the pairs that play the role of $xw$ must be very rare. Indeed, for any such pair of vertices $xw$ there must be many pairs of edges $e, e'$ for which $xw$ plays the same role. Thus, we would like to pack copies of $\mc{G}$ which can share at most two vertices, and if they do share two vertices, then the shared pair of vertices must be playing the role of $xw$ in both copies of $\mc{G}$. Thus, we say $xw$ is a {\em sharable} pair. We call the hypergraph $\mc{G}$, together with the specification that the only sharable pair is $xw$, a {\em gadget}. We would like to pack this gadget in an almost-perfect way so that almost every pair of vertices is contained in a hyperedge. The method Glock used in \cite{G19} was to glue many copies of $\mc{G}$ together by their sharable pair of vertices, and then pack copies of the resulting hypergraph $\mc{G}'$ such that edges from different copies are only allowed to share one vertex. Glock then packed copies of $\mc{G}'$ in a fairly standard way: he defined an appropriate auxiliary hypergraph in which the desired packing corresponds to an almost-perfect matching, and then applied the result of Pippinger and Spencer \cite{PS89}. The result in \cite{PS89} asserts that ``nice'' hypergraphs have almost-perfect matchings, and was also proved using a semi-random nibble method. This result (and its proof using a nibble method) is an important predecessor to the more recent work on conflict-free hypergraph matchings by Delcourt and Postle \cite{DP22} as well as Glock, Joos, Kim, K\"uhn and Lichev \cite{GJKKL}. This completes our proof sketch for Glock's result~\cite{G19}. 

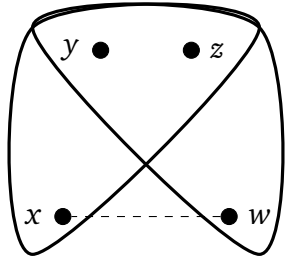
\begin{figure}[ht]
\begin{center}
\centering
\begin{tikzpicture}
	\node (x) at (-.1,0) [vert, label=left:$x$] {};
	\node (w) at (2.1,0) [vert, label=right:$w$] {};
	\node (y) at (.4,2.2) [vert, label=left:$y$] {};
	\node (z) at (1.6,2.2) [vert, label=right:$z$] {};
	
	\node (xsw) at (-.5,-.5) {};	
	\node (xse) at (.5,-.5) {};
	\node (ynw) at (-.5,2.5) {};	
	\node (zne) at (2.5,2.5) {};
	
	\node (wse) at (2.5,-.5) {};	

	\draw [dashed] (w) -- (x);

    \draw [very thick] plot [smooth cycle] coordinates {(xsw) (ynw) (zne)};
    \draw [very thick]plot [smooth cycle] coordinates {(wse) (zne) (ynw)};
\end{tikzpicture}
\end{center}
\caption{Glock's gadget $\mc{G}$.}
\label{fig:glock}
\end{figure}

Very recently, Glock, Joos, Kim, K\"uhn, Lichev and Pikhurko \cite{GJKKLP2022} proved Conjecture \ref{conj:BES} for the case $\ell=6$ by finding the limit. The proof is significantly more complicated but bears a resemblance to the sketch we provided above for Glock's proof \cite{G19} of the case $\ell=5$. In particular, the lower bound in \cite{GJKKLP2022} is obtained by packing an appropriately defined gadget, i.e.\ a hypergraph which has some sharable pairs. The upper bound follows from establishing a few constraints like \eqref{eqn:glock1} and \eqref{eqn:glock2} and then using linear programming (and the optimal solution to the linear program tells us what gadget to pack). However, to avoid all possible $(6, 4)$-configurations, one must avoid certain configurations consisting of several gadgets intersecting in certain ways. For this, a result like Pippinger and Spencer \cite{PS89} is not sufficient, and so they use conflict-free hypergraph matchings instead. 

Shortly after the $\ell=6$ case was proved in \cite{GJKKLP2022}, Delcourt and Postle \cite{DP22b} proved Conjecture~\ref{conj:BES} in full generality (but without actually finding the limit). Essentially, Delcourt and Postle~\cite{DP22b} proved that the sequence in Conjecture \ref{conj:BES} is ``approximately monotonic''. They did this by showing how to use near-optimal hypergraphs on say $n'$ vertices as gadgets that can be packed on $n \gg n'$ vertices. Even more recently, Glock, Kim, Lichev, Pikhurko and Sun \cite{GKLPS24} found the limit for the cases $\ell = 7, 8, 9$ of Conjecture \ref{conj:BES}, where all lower bounds followed by packing various gadgets. One of these gadgets has 63 vertices and 61 edges (but the rest are much smaller). 

For coloring problems such as estimating generalized Ramsey numbers, one can sometimes pack gadgets where the edges are colored. These gadgets generally have special rules about how they should be colored, which may depend on the other gadgets they intersect with. This is the idea behind our result with Pra\l at~\cite{BCDP22} for $f(n, 4, 5)$ as well as all the results in Joos and Mubayi \cite{JM22}. In \cite{BCD23}, the present authors showed how to obtain a near-optimal coloring for certain generalized Ramsey problems using a near-optimal hypergraph for an appropriate instance of the Brown-Erd\H{o}s-S\"os problem, which in some cases implies that we can find colored gadgets for the former problem by starting with gadgets for the latter. Very recently, Lane and Morrison \cite{LM24} and independently Bal, the first author, Heath and Zerbib \cite{BBHZ24} proved some more generalized Ramsey bounds by packing colored gadgets. To prove Theorem \ref{thm:main}, we will need a somewhat more complicated gadget than what was involved in these previous results. Our gadget will have sharable edges as well as some more complicated rules about how the colors must be chosen. 

\section{Lower bound}

Here we give an alternative proof that 
   $ f(n, 5, 8) \ge \frac 67 (n-1). $
    First we define the following. 
    \begin{definition}
        Given a (possibly only partial) coloring of the edges of $K_n$ and a set $S \subseteq V(K_n)$, suppose there are $x$ colored edges in $S$, having a total of $y\le x$ distinct colors. We say that $S$ has $x-y$ {\em color repetitions} or just {\em repetitions}. If $|S|=5$ and $S$ has 3 repetitions we say that $S$ is a $(5, 8)$-violation. A partial coloring with no $(5, 8)$-violation will be called a partial $(5, 8)$-coloring. Note that a partial coloring can be extended to a full $(5, 8)$-coloring if and only if it does not have any $(5, 8)$-violation .
    \end{definition}
    Consider a $(5, 8)$-coloring of $K_n$ using color set $C$. We will use this coloring to define a certain partition of the edges $E(K_n)$. We start with the partition $\mc{P}'$ into monochromatic components. In a $(5, 8)$-coloring a monochromatic component has at most three edges. For the rest of this proof whenever we say ``component'' we mean a monochromatic component. It is sufficient to assume that there are no monochromatic triangles. If there were a monochromatic (blue) triangle, each edge incident to this triangle would be an isolated edge in $\mathcal P'$ and no other edge could be blue. Thus, swapping the colors of an edge of the triangle and an incident edge would yield another $(5,8)$-coloring with the same number of colors without the monochromatic triangle. Therefore, each part of $\mc{P}'$ is either an isolated edge, 2-path, 3-path, or 3-star. For each part $P\in \mc{P}'$ consisting of a component in color $c$, we say $P$ {\em hits} each pair $(v, c)$ where $v$ is a vertex in the component. Clearly, each pair $(v, c)$ is hit by at most one part. 

    We will form a new partition $\mc{P}$ by starting with $\mc{P}'$ and then applying a sequence of operations where we take several parts of the current partition, remove them and replace them with their union (i.e.~ {\em merge} the parts). When we merge a set of parts into a new part $P$, we say $P$ hits all the pairs $(v, c)$ that were hit by any of the merged parts. We will also in some cases specify some additional pairs we will say are hit by $P$. It is important for our argument (and we will see it holds) that each pair $(v, c)$ is still hit by at most one part.

    Now we start to describe the merging. First, for any part $P =\{xy, xv, xu\} \in \mc{P}'$, i.e.\ a 3-star component, we see that $\mc{P}'$ also has singleton parts $\{uv\}, \{uy\}, \{vy\}$ (if any of these edges were in a larger component we would have a $(5, 8)$-violation). See Figure~\ref{fig:P}(1). We merge $P$ with these singletons forming a new part $Q$ with six edges. Note that two distinct 3-stars can share at most one vertex  since otherwise we have a $(5, 8)$-violation. Thus each singleton part in $\mc{P}'$ gets merged with at most one 3-star. In addition to all the 10 pairs hit by the parts we merged to form $Q$, we also say that $Q$ hits $(x, \text{green}), (x, \text{red}), (x, \text{orange})$ (see Figure~\ref{fig:P}(1)). Any scenario in which a pair is hit twice implies a $(5, 8)$-violation. Indeed, in such a scenario, we would be reusing some color used in $\{x,y,u,v\}$ and introducing one additional vertex, resulting in 3 repetitions on 5 vertices. Next, we merge all 3-stars with their respective singletons to form a new partition. We abuse notation and call the new partition $\mc{P}'$, and we will keep doing so after even more merging until we reach our final partition $\mc{P}$. 

    Next, for any part $P =\{ux, xy, yv\} \in \mc{P}'$, i.e.\ a 3-path component, we see that $\mc{P}'$ also has singleton parts $\{uy\}, \{xv\}$ (and also $\{uv\}$, but we don't need that for this merging). See Figure~i.e.\ \ref{fig:P}(2). We merge $P$ with these singletons forming a new part $Q$ with five edges. We say the new part hits the additional pairs $(x, \text{green}), (y, \text{red})$. Next, for any pair of parts $P=\{xu, uv\}, P'=\{xy, yv\}\in \mc{P}'$, i.e.\ a pair of 2-paths sharing endpoints, we will merge $P, P'$ and the singleton part $\{xv\}$. See Figure~i.e.\ \ref{fig:P}(3). We say the new part hits the additional pairs $(u, \text{green}), (y, \text{green})$. Next, for any pair of parts $P=\{xy, yu\}, P'=\{xu, uv\}\in \mc{P}'$ we will merge $P, P'$ and the singleton part $\{xv\}$. See Figure~i.e.\ \ref{fig:P}(4). We say the new part hits the additional pair $(u, \text{green})$.

    Next, for any remaining part $P=\{ux, xv\} \in \mc{P}'$, i.e.\ a 2-path which did not get merged with anything yet, we will do the following. First note that $\{uv\}$ must be a singleton, say in the color blue. If $x$ is not incident to any blue edge then we just merge $P$ with $\{uv\}$ and say the new part hits the additional pair $(x, \text{blue})$. See Figure  i.e.\ \ref{fig:P}(5). Otherwise let $xy$ be the (necessarily unique) blue edge incident with $x$. Now if there is no blue edge $zw$ such that the two edges $yz, yw$ are the same color (like in Figure~i.e.\ \ref{fig:P}(7)), then we will merge the three parts $P=\{ux, xv\}, \{uv\}, \{xy\}$. See Figure~i.e.\ \ref{fig:P}(6). This new part does not hit any additional pairs. Finally in the case where there is a (necessarily unique) blue edge $zw$ such that the two edges $yz, yw$ are the same color, we merge the six parts $P=\{ux, xv\}, \{wy, yz\}, \{uv\}, \{xy\}, \{zw\}$. This new part does not hit any additional pairs. After completing all these merging operations we arrive at our partition $\mc{P}$. Every part of  $\mc{P}$ is isomorphic (where an ``isomorphism'' allows permuting colors) to one of the subfigures in Figure~\ref{fig:P}.

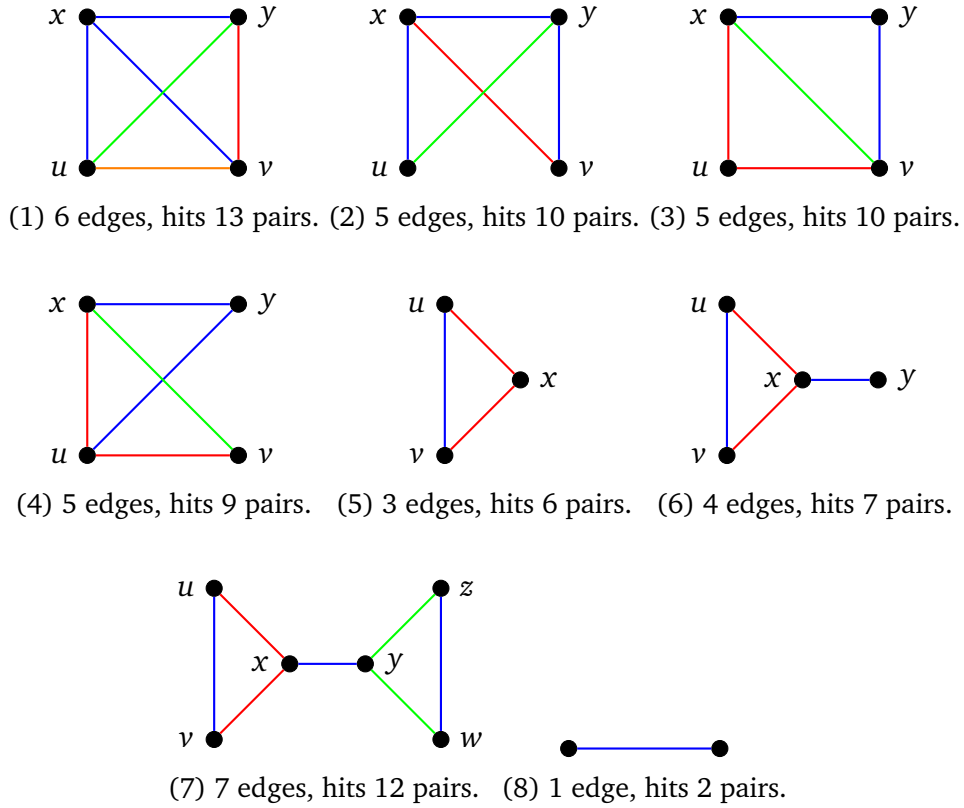
\begin{figure}[ht]
\begin{center}
\begin{subfigure}[b]{.25\textwidth}
\centering
\begin{tikzpicture}
	\node (u) at (0,0) [vert, label=left:$u$] {};
	\node (v) at (2,0) [vert, label=right:$v$] {};
	\node (x) at (0,2) [vert, label=left:$x$] {};
	\node (y) at (2,2) [vert, label=right:$y$] {};
	\draw [thick,blue] (x) -- (v);
	\draw [thick,blue] (x) -- (u);
    \draw [thick, blue] (x) -- (y);
    \draw [thick, red] (y) -- (v);
    \draw [thick, green] (y) -- (u);
    \draw [thick, orange] (v) -- (u);
\end{tikzpicture}
\caption{6 edges, hits 13 pairs.}
\vspace{3ex}
\end{subfigure}
\begin{subfigure}[b]{.25\textwidth}
\centering
\begin{tikzpicture}
	\node (u) at (0,0) [vert, label=left:$u$] {};
	\node (v) at (2,0) [vert, label=right:$v$] {};
	\node (x) at (0,2) [vert, label=left:$x$] {};
	\node (y) at (2,2) [vert, label=right:$y$] {};
	\draw [thick,red] (x) -- (v);
	\draw [thick,blue] (x) -- (u);
    \draw [thick,blue] (x) -- (y);
    \draw [thick,blue] (y) -- (v);
    \draw [thick,green] (y) -- (u);
\end{tikzpicture}
\caption{5 edges, hits 10 pairs.}
\vspace{3ex}
\end{subfigure}
\begin{subfigure}[b]{.25\textwidth}
\centering
\begin{tikzpicture}
	\node (u) at (0,0) [vert, label=left:$u$] {};
	\node (v) at (2,0) [vert, label=right:$v$] {};
	\node (x) at (0,2) [vert, label=left:$x$] {};
	\node (y) at (2,2) [vert, label=right:$y$] {};
	\draw [thick,red] (u) -- (v);
	\draw [thick,red] (u) -- (x);
    \draw [thick,blue] (x) -- (y);
    \draw [thick,blue] (y) -- (v);
    \draw [thick,green] (x) -- (v);
\end{tikzpicture}
\caption{5 edges, hits 10 pairs.}
\vspace{3ex}
\end{subfigure}
\begin{subfigure}[b]{.25\textwidth}
\centering
\begin{tikzpicture}
	\node (u) at (0,0) [vert, label=left:$u$] {};
	\node (v) at (2,0) [vert, label=right:$v$] {};
	\node (x) at (0,2) [vert, label=left:$x$] {};
	\node (y) at (2,2) [vert, label=right:$y$] {};
	\draw [thick,red] (u) -- (v);
	\draw [thick,red] (u) -- (x);
    \draw [thick,blue] (x) -- (y);
    \draw [thick,blue] (y) -- (u);
    \draw [thick,green] (x) -- (v);
\end{tikzpicture}
\caption{5 edges, hits 9 pairs.}
\vspace{3ex}
\end{subfigure}
\begin{subfigure}[b]{.25\textwidth}
\centering
\begin{tikzpicture}
    \node (u) at (0,2) [vert, label=left:$u$] {};
	\node (v) at (0,0) [vert, label=left:$v$] {};
	\node (x) at (1,1) [vert, label=right:$x$] {};
	\draw [thick,blue] (u) -- (v);
	\draw [thick,red] (u) -- (x);
 	\draw [thick,red] (x) -- (v);
\end{tikzpicture}
\caption{3 edges, hits 6 pairs.}
\vspace{3ex}
\end{subfigure}
\begin{subfigure}[b]{.25\textwidth}
\centering
\begin{tikzpicture}
	\node (u) at (0,2) [vert, label=left:$u$] {};
	\node (v) at (0,0) [vert, label=left:$v$] {};
	\node (x) at (1,1) [vert, label=left:$x$] {};
	\node (y) at (2,1) [vert, label=right:$y$] {};
	\draw [thick,red] (x) -- (u);
	\draw [thick,red] (x) -- (v);
    \draw [thick,blue] (u) -- (v);
    \draw [thick,blue] (y) -- (x);
\end{tikzpicture}
\caption{4 edges, hits 7 pairs.}
\vspace{3ex}
\end{subfigure}
\begin{subfigure}[b]{.25\textwidth}
\centering
\begin{tikzpicture}
	\node (u) at (0,2) [vert, label=left:$u$] {};
	\node (v) at (0,0) [vert, label=left:$v$] {};
	\node (x) at (1,1) [vert, label=left:$x$] {};
	\node (y) at (2,1) [vert, label=right:$y$] {};
	\node (z) at (3,2) [vert, label=right:$z$] {};
	\node (w) at (3,0) [vert, label=right:$w$] {};
	\draw [red] (x) -- (u);
	\draw [red] (x) -- (v);
    \draw [blue] (u) -- (v);
    \draw [blue] (y) -- (x);
    \draw [green] (z) -- (y);
    \draw [green] (y) -- (w);
    \draw [blue] (z) -- (w);
\end{tikzpicture}
\caption{7 edges, hits 12 pairs.}
\end{subfigure}
\begin{subfigure}[b]{.25\textwidth}
\centering
\begin{tikzpicture}
	\node (u) at (0,0) [vert] {};
	\node (v) at (2,0) [vert] {};
	\draw [blue] (u) -- (v);
\end{tikzpicture}
\caption{1 edge, hits 2 pairs.}
\end{subfigure}
\end{center}
\caption{Representations of all possible parts of $\mc{P}$.}
\label{fig:P}
\end{figure}

For $i=1, \ldots, 8$ let $x_i$ be the number of parts in $\mc P$ that look like Figure~\ref{fig:P}$(i)$.
Then summing the edges we have
\begin{equation*}
    6x_1 + 5 x_2 + 5x_3 + 5x_4+ 3x_5 + 4x_6 + 7x_7 + x_8 = \binom n2.
\end{equation*}

Let $H_i$ be the subset of all pairs $(v,c)$ that are hit by a part of $\mc{P}$ of type $(i)$ in Figure~\ref{fig:P}. Notice that the total number of pairs $(v,c)$ hit is $|\cup_{i=1}^8 H_i| \le n|C|$. Since each pair $(v, c)$ is hit by at most one part of $\mc{P}$, $H_i \cap H_j = \emptyset$ for each $i \not = j$ and so the total number of $(v,c)$ hit is just the sum of the $|H_i|$ for $i=1,\ldots,8$. Further, $|H_i|$ can be found by referring to Figure~\ref{fig:P}: $|H_1|=13x_1, \ldots |H_8|=2x_8$. Thus,
\begin{equation}\label{eqn:lowerbound1}
13x_1 + 10x_2 + 10x_3 + 9x_4 + 6x_5 + 7x_6 + 12x_7 + 2x_8 \le n|C|.
\end{equation}
But then multiplying the first equation by $12/7$ and subtracting from the second one yields
\begin{align}
    n|C| - \frac{12}{7} \binom n2&\ge 13x_1 + 10x_2 + 10x_3 + 9x_4 + 6x_5 + 7x_6 + 12x_7 + 2x_8\nn\\
    & \qquad - \frac{12}{7} \left( 6x_1 + 5 x_2 + 5x_3 + 5x_4+ 3x_5 + 4x_6 + 7x_7 + x_8\right) \nn\\
    & = \frac{19}{7} x_1 + \frac{10}{7} x_2 + \frac{10}{7} x_3 + \frac{3}{7} x_4 + \frac{6}{7} x_5 + \frac{1}{7} x_6 + \frac{2}{7} x_8 \ge 0.\label{eqn:lowerbound2}
\end{align}
Thus, $|C| \ge \frac 1n \cdot \frac{12}{7} \binom n2 = \frac 67 (n-1)$, as required.

\section{Preliminaries for the upper bound}

\subsection{Intuition}
In this section we describe the construction for the upper bound in Theorem \ref{thm:main}.
In order to come close to the lower bound, the inequalities in lines \eqref{eqn:lowerbound1} and \eqref{eqn:lowerbound2} must be nearly tight. For \eqref{eqn:lowerbound2} to be nearly tight means that $x_7$ (the only variable absent from the line above \eqref{eqn:lowerbound2}) must be quadratic and the rest of the $x_i$ must be subquadratic, i.e.\ almost every part in the partition $\mathcal{P}$ is of the type shown in Figure~\ref{fig:P}(7). This gives us the idea that a good coloring might be obtained by ``packing'' sets of colored edges that look like Figure~\ref{fig:P}(7). And so we have essentially derived our gadget! However we will have to pack this gadget carefully as we will see. We call the colored graph in Figure~\ref{fig:P}(7) the \emph{gadget}, where the colors red, blue and green can be permuted or swapped with any other colors. 

We will need to take care how the gadgets in our packing intersect. We need almost all the edges of $K_n$ to be covered by edge-disjoint gadgets, and if we are careless we might end up with two gadgets intersecting as shown in Figure~\ref{fig:intersect}(2). In that figure the red and blue edges come from one gadget, and the orange and gray from a second gadget. If the two gadgets intersect as shown we see a $(5, 8)$-violation. To avoid this kind of complication we will only allow a pair of gadgets to share a pair of vertices if that pair is a nonedge in both gadgets, as shown in Figure~\ref{fig:intersect}(3). In order to ensure we can still cover almost all edges of $K_n$ with our gadgets, we will need to make sure that there are not too many of these pairs of vertices that get used as nonedges of gadgets (such pairs will never be covered by our packing). Thus, before we try to get our packing, we will designate a small set of edges to be our {\it shareable edges} (see Figure~\ref{fig:intersect}(1)). When we pack a gadget, we will make sure the non-edges of the gadget are all shareable, and the edges of the gadget are not shareable. This way we can allow two gadgets to share two vertices as long as this pair of vertices is a shareable edge and therefore a non-edge in each of the gadgets. We will call this gadget in Figure~\ref{fig:intersect}(1) $\Gamma$. We will define $\Gamma$ more formally later, but it will be a set of edges colored according to the pattern shown in Figure~\ref{fig:intersect}(1) but with the colors possibly permuted, with the dotted edges always being sharable. The formal definition of $\Gamma$ will also involve additional rules about how the colors are chosen, which we start to describe next. 

\begin{figure}[ht]
\begin{center}
\begin{subfigure}[b]{.31\textwidth}
\centering
\begin{tikzpicture}
\begin{scope}[scale=1.8,yscale=1,rotate=90, shift={(0,0)}]
	\DrawGammaSpecial{blue}{red}{green}{}
\end{scope}
\end{tikzpicture}
\caption{A gadget with shareable edges shown as dashed lines.}
\vspace{3ex}
\end{subfigure}
\hfill
\begin{subfigure}[b]{.31\textwidth}
\centering
\begin{tikzpicture}
\begin{scope}[scale=.9, rotate=-90]
\begin{scope}[shift={(0,0)}, rotate=90]
	\DrawGamma{blue}{red}{green}{}
\end{scope}
\begin{scope}[scale=1.93185,shift={(-1.45,.13)}, rotate=90+75]
	\DrawGamma{gray}{orange}{cyan}{}
\end{scope}
\end{scope}
\end{tikzpicture}
\caption{A forbidden intersection of gadgets.}
\vspace{3ex}
\end{subfigure}
\hfill
\begin{subfigure}[b]{.31\textwidth}
\centering
\begin{tikzpicture}
	\begin{scope}[scale=1.4, rotate=0]
    	\begin{scope}[shift={(0,0)}, rotate=90]
    		\DrawGamma{blue}{red}{green}{}
    	\end{scope}
    	\begin{scope}[scale=1,shift={(.5,.87)}, rotate=90]
    		\DrawGamma{gray}{orange}{cyan}{}
    	\end{scope}
	\end{scope}
\end{tikzpicture}
\caption{An allowed intersection of gadgets.}
\vspace{3ex}
\end{subfigure}
\end{center}
\caption{Intersecting gadgets.}
\label{fig:intersect}
\end{figure}
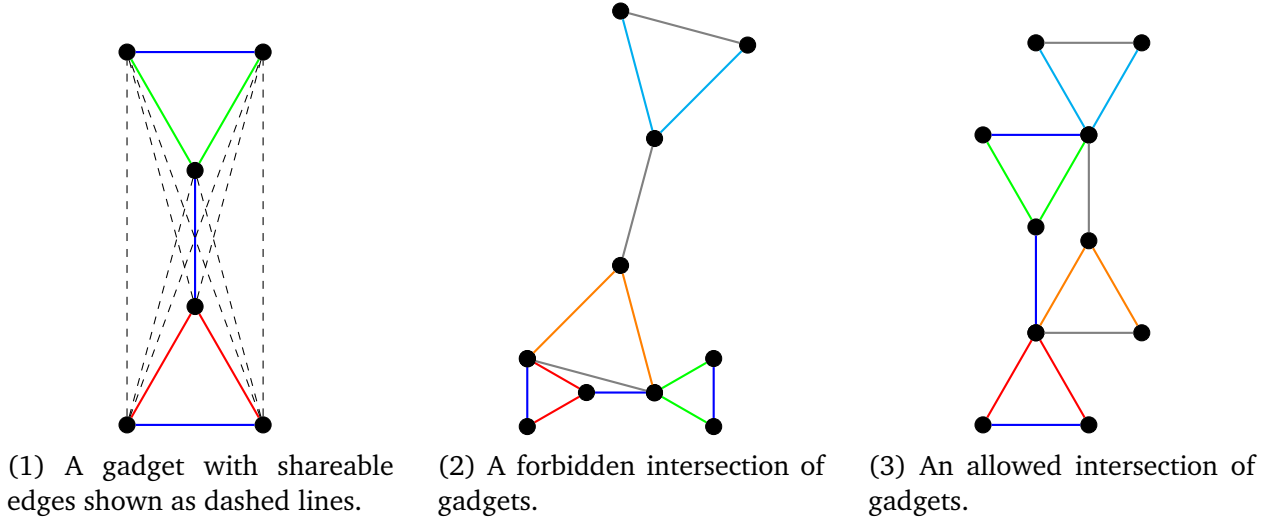

We still need to guarantee that \eqref{eqn:lowerbound1} is nearly tight, meaning that almost every pair in $V(K_n) \times C$ is hit. Recall (see Figure~\ref{fig:P}(7)) that each gadget hits 12 pairs. If we are to use this gadget in our packing though, we must make sure that $y$ never touches any red edge from another gadget or else we will not get a $(5, 8)$-coloring. In order to guarantee that we can still hit almost every pair, we must make sure there are not too many colors that play this role for $y$ (i.e.\ colors that appear in a gadget involving $y$ but not actually touching $y$, like the color red in Figure~\ref{fig:P}(7)). Thus, before we try to get our packing we designate for each vertex a small set of {\it non-touching colors} which are allowed to play this role. Whenever we pack a gadget using the vertex $y$, the colors that do not touch $y$ will be chosen from $y$'s non-touching colors, and the colors that do touch $y$ will be from among $y$'s touching colors (i.e.\ the rest of the colors). 

We will obtain a $(5, 8)$-coloring of $K_n$ via some tools that were developed very recently by Glock, Joos, Kim, K\"uhn and Lichev \cite{GJKKL} and independently by Delcourt and Postle \cite{DP22}. The two versions of these tools in \cite{DP22, GJKKL} are similar but different, and here we actually use both of them.

\subsection{Tools and notation}

We use many of the definitions from Joos and Mubayi \cite{JM22}. 

Let $\mc{H}$ be a hypergraph. If all the edges of $\mc{H}$ have size $k$ we say $\mc{H}$ is \emph{$k$-uniform}, and if all edges have size at most $k$ we say $\mc{H}$ is \emph{$k$-bounded}. A \emph{matching} in $\mc{H}$ is a set of pairwise disjoint edges. We denote by $\mc{H}^{(j)}$ the hypergraph on the same vertex set as $\mc{H}$ and with all the edges of size $j$.
For $v \in V(\mc{H})$, we denote by $\mc{H}_v$ the set of all $e\setminus\{v\}$ such that  $v \in e\in \mc{H}$.
For a hypergraph $\mc{H}$ and a vertex $v \in V(\mc{H})$, 
we denote by $d_\mc{H}(v)$ the \emph{degree} of $v$; 
that is, the number of edges containing $v$.
We use $\Delta(\mc{H})=\max_{u\in V(\mc{H})}d_\mc{H}(u)$ to denote the \emph{maximum degree} of $\mc{H}$.
Similarly, we define the \emph{minimum degree} $\delta(\mc{H})$ of $\mc{H}$.
For $j\geq 2$, we denote by $\Delta_j(\mc{H})$ the maximum number of edges that contain a particular set of $j$ vertices.

We say that $\mc{C}$ is a \emph{conflict system} for a hypergraph $\mc{H}$ if $V(\mc{C}) = E(\mc{H})$. 
In this case, we call the edges of $\mc{C}$ \emph{conflicts}. We say that a matching $\mc{M}\subseteq E(\mc{H})$ is \emph{$\mc{C}$-free} if $\mc{M}$ does not contain any $C\in \mc{C}$ as a subset.

We will use the following adjusted version of the theorem that appears in Joos and Mubayi's paper \cite{JM22}. It is a simplified version of Theorem~3.3 in \cite{GJKKL}.

\begin{theorem}[{\cite[Theorem~2.1]{JM22}}]\label{thm:A}
	For all $k,\ell\geq 2$, there exists $\eps_0>0$ such that for all $\eps\in(0,\eps_0)$, there exists $d_0$ such that the following holds for all $d\geq d_0$.
	Suppose $\mc{H}$ is a $k$-uniform hypergraph on $n\leq \exp(d^{\eps^3})$ vertices and suppose the following conditions are satisfied:
 \begin{enumerate}[label=(H\arabic*A)]
\item $\Delta_2(\mc{H}) \le d^{1-\eps}$, and \label{cond:h1a}
\item $(1-d^{-\eps})d\leq \delta(\mc{H})\leq\Delta(\mc{H})\leq d$. \label{cond:h2a}
\end{enumerate}
Suppose further that $\mc{C}$ is an $\ell$-bounded conflict system for $\mc{H}$ such that $|C|\geq 3$ for all $C\in \mc{C}$, and satisfying the following conditions:
 \begin{enumerate}[label=\textup{(C\arabic*A)}]
		\item\label{cond:c1a} $\Delta(\mc{C}^{(j)})\leq \ell d^{j-1}$ for all $3\leq j \leq \ell$, and 
		\item\label{cond:c2a} $\Delta_{j'}(\mc{C}^{(j)})\leq d^{j-j'-\eps}$ for all $2 \le j' < j \le \ell$.
\end{enumerate}
Then, there exists a $\mc{C}$-free matching $\mc{M}\subset \mc{H}$ of size at least $(1-d^{-\eps^3})n/k$
\end{theorem}

For a hypergraph $\mc{H}$ and a partition $V(\mc{H}) = X \cup Y$, we say that $\mc{H}$ is \emph{bipartite} and has \emph{bipartition} $X \cup Y$ if every edge of $\mc{H}$ contains exactly one vertex of $X$. In this case we say that a matching $\mc{M}$ in $\mc{H}$ is $X$-perfect if every vertex of $X$ is contained in some edge in $\mc{M}$. 

\begin{theorem}[{\cite[Theorem~1.16]{DP22}}]\label{thm:B}
For all  $k, \ell \ge 2$ and real $\eps \in  (0, 1)$, there exist an integer $d_0 > 0$ and real $\alpha > 0$ such that following holds for all $d \ge d_0$. Let $\mc{H}$ be a bipartite $k$-bounded hypergraph  with bipartition $X \cup Y$ such that
\begin{enumerate}[label=(H\arabic*B)]
\item $\Delta_2(\mc{H}) \le d^{1-\eps}$, and \label{cond:h1b}
\item every vertex in $X$ has degree at least $(1+d^{-\alpha})d$ and every vertex in $Y$ has degree at most $d$. \label{cond:h2b}
\end{enumerate}
Suppose further that $\mc{C}$ is an $\ell$-bounded conflict system for $\mc{H}$ satisfying the following conditions:
\begin{enumerate}[label=(C\arabic*B)]
\item  $\Delta(\mc{C}^{(j)}) \le \alpha \cdot d^{j-1} \log d$ for all $2\le j \le \ell$, \label{cond:c1b}
\item $\Delta_{j'}(\mc{C}^{(j)}) \le d^{j-j'-\eps}$ for all $2 \le j' < j \le \ell$, \label{cond:c2b}
\item for all $v \in V(\mc{H})$ and $e \in E(\mc{H})$ such that $v \notin e$, we have
\[
\big|\big\{e' \in E(\mc{H}) : e' \ni v \mbox{ and } \{e, e'\}\in \mc{C}\big\}\big| \le d^{1-\eps}, \mbox{  and}
\]\label{cond:c3b}
\item for all $e, e' \in E(\mc{H})$ we have
\[
\big|\big\{e'' \in E(\mc{H}) : \{e, e''\}\in \mc{C} \mbox{ and } \{e', e''\}\in \mc{C}\big\}\big| \le d^{1-\eps}.
\] \label{cond:c4b}
\end{enumerate}
Then there exists an $X$-perfect $\mc{C}$-free matching of $\mc{H}$.
\end{theorem}

We will also use McDiarmid's inequality. 

\begin{theorem}[McDiarmid’s inequality~\cite{M89}]\label{thm:mcd}
Suppose $X_1,\ldots, X_m$ are independent random variables. Suppose X is a real-valued random variable determined by $X_1,\ldots, X_m$ such that changing the outcome of $X_i$ changes $X$ by at most $b_i$ for all $i \in [m]$. Then, for all $t > 0$, we
have
$$
P\left(|X-\E[X]| \ge t\right) \le 2 \exp\left( -\frac{2t^2}{\sum_{i \in [m]}b_i^2}\right).
$$
\end{theorem}

\subsection{Overview of the coloring}\label{sec:overview}

Our construction will start by identifying our set of shareable edges, and non-touching colors for each vertex. Then we would like to color almost all the edges (excluding shareable edges and a few other edges) by packing edge-disjoint gadgets. These gadgets must be packed to obey the rules for shareable edges and the colors chosen to obey the rules for non-touching colors. We will use two disjoint sets of colors $C_{A1}$ and $C_{A2}$ for these gadgets, and the colors in $C_{A1}$ and $C_{A2}$ will play different roles. For $c \in C_{A1}$, $c$ will only be allowed to play the role of blue in Figure~i.e.\ \ref{fig:P}(7), and so the color class for $c$ will consist of vertex-disjoint 1-paths (i.e.\ a matching in $K_n$). For $c \in C_{A2}$, $c$ will only be allowed to play the role of red or green in Figure~i.e.\ \ref{fig:P}(7), so the color class for $c$ will consist of vertex-disjoint 2-paths. In fact, to get our asymptotically optimal bound on the number of colors we will need each color in $C_{A1}$ to induce an almost-perfect matching and each color in $C_{A2}$ to induce a set of 2-paths covering almost all vertices. When we pack the gadgets the colors will also be chosen to obey certain other rules to avoid $(5, 8)$-violations. This part of our construction will be called {\bf Phase A}, and the proof that it works will be an application of Theorem \ref{thm:A}. In particular we will define an auxiliary hypergraph $\mc{H}_A$ and conflict system $\mc{C}_A$ such that a $\mc{C}_A$-free matching in $\mc{H}_A$ corresponds to a partial $(5, 8)$-coloring. The advantage of using Theorem \ref{thm:A} here is that this theorem has some extensions which we will use (see Section \ref{sec:quasirandom}). These extensions allow us to prove that the coloring obtained in Phase A has some nice quasirandom properties which will allow us to do the next phase. 
In {\bf Phase B} we will color all edges that remain uncolored after Phase A. This will be done by applying Theorem \ref{thm:B} to a new hypergraph $\mc{H}_B$ and conflict system $\mc{C}_B$. 

It helps to think of our construction as the output of a random procedure. Indeed, Theorem \ref{thm:A} was proved using a random greedy process that forms a matching by choosing one random edge at a time. Likewise, Theorem \ref{thm:B} was proved using a semirandom ``nibble'' method where many random edges are chosen at a time. 

Both Theorems \ref{thm:A} and \ref{thm:B} have been applied to coloring problems before \cite{BCD23, BDLP22, BHZ23, DP22, GHPSZ23, JM22}. To get a coloring, one  defines an appropriate auxiliary hypergraph where some vertices correspond to the objects being colored and other vertices correspond to colors. Each edge in this hypergraph will contain a few objects and a few colors, and it is interpreted as assigning those colors to those objects. The objects we want to color are the edges of $K_n$, and in Phase A we will color them in ``groups'' (i.e.\ sets of edges forming a gadget). Neither Theorem \ref{thm:A} nor \ref{thm:B} alone lends itself to coloring all of our edges this way, so we will settle for almost all of them. Theorem \ref{thm:A} is better here since it has the extensions we discussed in the previous paragraph. In Phase B we will color the edges individually (i.e.\ not in ``groups'' as in Phase A). Theorem \ref{thm:B} is capable of coloring all the remaining edges with no leftovers. In particular in our application of Theorem \ref{thm:B}, the vertex set $X$ will be the set of uncolored edges and $Y$ will represent colors for those edges. The $X$-perfect matching guaranteed by Theorem \ref{thm:B} will then correspond to a coloring of all remaining edges. The fact that we color the edges individually in Phase B is exactly why we can encode it specifically as a bipartite graph and apply Theorem \ref{thm:B}. 

In several of the recent results on Generalized Ramsey numbers and similar coloring problems \cite{BCD23,BCDP22,  BDE22, BHZ23, DP22, GHPSZ23, JM22}, colorings were also obtained using a two-phase construction where the first phase was a random greedy process or nibble method (either explicitly or implicitly as an application of Theorem \ref{thm:A} or \ref{thm:B}). The second phase in these applications was an application of the Local Lemma. We would have liked to do our Phase B with the Local Lemma too, but as we will see it would not work at least in the straightforward way.  Our conflict system for Phase B is too dense for a naive application of the Local Lemma. Fortunately we have Theorem \ref{thm:B} which can handle the situation. It is worth pointing out that the proof of Theorem \ref{thm:B} is also a kind of two-phase construction: the first phase is a nibble method to get a matching that covers almost all of $X$, and for the second phase each remaining vertex of $X$ chooses a random edge and the Local Lemma is used to show this gives a matching with positive probability. So in the end we are implicitly using the Local Lemma by applying Theorem \ref{thm:B}.

\section{Phase A}

\subsection{Setting up our application of Theorem \ref{thm:A}}

We would like to prove the following theorem (actually we will eventually prove something stronger, see Section \ref{sec:quasirandom}). We will use Theorem \ref{thm:A}.

\begin{theorem}\label{thm:phaseA}
 There is a partial $(5, 8)$-coloring of $K_n$ using at most $\frac 67 n + o(n)$ colors and coloring all but $o(n^{2})$ edges.
\end{theorem}

\begin{proof}
 To get the coloring we need, we will  apply Theorem \ref{thm:A} to a carefully constructed hypergraph $\mc{H}_A$ and conflict system $\mc{C}_A$ which we will now start to describe. First we will generate a random set of edges $E' \subseteq E(K_n)$ by removing each possible edge with probability $p:=n^{-\delta}$ independently, for some $\d>0$.  The set of removed edges $E'' := E(K_n) \setminus E'$ will be our shareable edges. Let $C_A$ be a set of $(1-p/2) (6/7)n = \frac 67 n + O(n^{1-\d})$ colors. We split the color set $C_A$ into disjoint sets $C_{A1}$ and $C_{A2}$ where 
\[
|C_{A1}| = (1-p)\frac{3}{7}n \quad \text{and} \quad |C_{A2}| = \frac 37 n. 
\]

We will apply Theorem \ref{thm:A} with $\eps = \eps_A$, and at certain points of our proof we will see inequalities that $\d, \eps_A$ must satisfy. We will say $\d, \eps_A$ are chosen such that 
 \begin{equation}\label{eqn:ed}
     0< \eps_A \ll \d \ll 1.
 \end{equation}
 In other words, if we first choose $\d>0$ small enough, we can then choose $\eps_A>0$ small enough with respect to our choice for $\d$.

We define an auxiliary set of vertices 
\[
V:=\{(v, c): v \in V(K_n), c \in C_A\}.
\]
We let $V' \subseteq V$ be a random subset where for each vertex $(v, c)$ with $c \in C_{A2}$, we remove $(v, c)$ with probability $p$. Removing $(v, c)$ will mean that $c$ is one of the non-touching colors for $v$. 

We define the following graph $\Gamma^*$ (which is just the gadget without colors): $$V(\Gamma^*)=\{\gamma^*_i\}_{i=1, \ldots, 6}  \quad \text{and} \quad E(\Gamma^*)=\bigg\{\gamma^*_1\gamma^*_2,\; \gamma^*_1\gamma^*_3,\; \gamma^*_2\gamma^*_3,\; \gamma^*_3\gamma^*_4,\; \gamma^*_4\gamma^*_5,\; \gamma^*_5\gamma^*_6,\; \gamma^*_4\gamma^*_6 \bigg\}.$$

We are ready to define $\mc{H}_A$. The vertex set will be 
\begin{equation}\label{eqn:VHA}
  V(\mc{H}_A) =  E' \cup V'.  
\end{equation}
Thus for our application of Theorem \ref{thm:A} we will use 
\begin{equation}\label{eqn:nadef}
    n_A := |V(\mc{H}_A)| = |E'|+|V'| = (1+o(1))\rbrac{\binom n 2 + n \cdot \frac 67 n } = (1+o(1))\frac{19}{14}n^2
\end{equation}
(the estimate above holds w.h.p.~when we randomly generate $E', V'$).
$\mc{H}_A$ will be $k_A$-uniform where $k_A=19$, and each edge will have 7 elements from $E'$ and 12 elements from $V'$. There will be a one-to-one correspondence between the edges of this hypergraph and colored copies of $\Gamma^*$ satisfying certain conditions below. 
More precisely $E(\mc{H}_A)$ will consist of all edges of the form
\[
E(\Gamma) \cup \bigg\{ (\gamma_i, c_1)  \bigg\}_{i =1,\ldots, 6}\cup \bigg\{ (\gamma_i, c_2)  \bigg\}_{i = 1,2,3} \cup \bigg\{ (\gamma_i, c_2')  \bigg\}_{i = 4,5,6}
\]
where $\Gamma\subseteq K_n$ is isomorphic to $\Gamma^*$ with each vertex $\gamma_i$ corresponding to $\gamma_i$, such that the following are satisfied:
\begin{enumerate}
    \item \label{item:Hcond1} $E(\Gamma) \subseteq E'$,
    \item \label{item:Hcond2}$E\left(\overline{\Gamma}\right) \subseteq  E(K_n) \setminus E'$, where $\overline{\Gamma}$ is the complement of $\Gamma$,
    \item \label{item:Hcond3}$c_1 \in C_{A1}$ and  $c_2, c_2' \in C_{A2}$,
    \item \label{item:Hcond4}$\bigg\{ (\gamma_i, c_2)  \bigg\}_{i = 1,2,3} \cup \bigg\{ (\gamma_i, c_2')  \bigg\}_{i = 4,5,6} \subseteq V'$
    \item \label{item:Hcond5}$(\gamma_i, c_2') \notin V'$ for $i=1,2,3$ and $(\gamma_i, c_2) \not \in V'$ for $i = 4,5,6$.
\end{enumerate}
\begin{figure}[ht]
\begin{center}
    \begin{tikzpicture}[scale=1.5]
	\node (g1) at (-1.37,.5) [vert, label=above:$\gamma_1$] {};
	\node (g2) at (-1.37,-.5) [vert, label=below:$\gamma_2$] {};
	\node (g3) at (-.5,0) [vert, label=above:$\gamma_3$] {};
	\node (g4) at (.5,0) [vert, label=above:$\gamma_4$] {};
	\node (g5) at (1.37,.5) [vert, label=above:$\gamma_5$] {};
	\node (g6) at (1.37,-.5) [vert, label=below:$\gamma_6$] {};
	
	\draw [thick,blue] (g1) -- node[left] {$c_1$} (g2) ;
	
	\draw [thick,red] (g1) -- node[above] {$c_2$} (g3) ;
	\draw [thick,red] (g2) -- node[below] {$c_2$} (g3) ;
	
	\draw [thick,blue] (g3) -- node[above] {$c_1$} (g4) ;
	
	\draw [thick,green] (g4) -- node[above] {$c_2'$} (g5) ;
	\draw [thick,green] (g4) -- node[below] {$c_2'$} (g6) ;
	
	\draw [thick,blue] (g5) -- node[right] {$c_1$} (g6) ;
\end{tikzpicture}
\caption{A colored copy of $\Gamma$, corresponding to an edge of $\mc{H}_A$.}
\label{fig:G}
\end{center}
\end{figure}
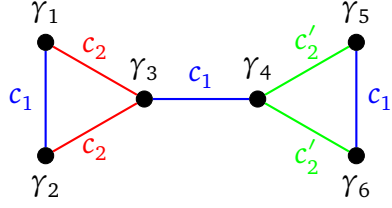

We call the edges of $\mc{H}_A$ {\em gadgets}. Each gadget corresponds to (and we will identify it with) a colored copy of $\Gamma^*$ in $K_n$. The definition of our hypergraph $\mc{H}_A$ is an encoding of these gadgets designed so that a matching in $\mc{H}_A$ will correspond to a partial coloring of $K_n$ which completely avoids many types of $(5, 8)$-violations. However this coloring might still contain other $(5, 8)$-violations, which we must avoid by using our conflict system $\mc{C}_A$, which we define next.

The conflict system $\mc{C}_A$ for $\mc{H}_A$ of course has vertex set $E(\mc{H}_A)$, i.e.\ the set of gadgets. The conflicts (edges) of $\mc{C}_A$ will be certain sets of gadgets which together would create a $(5, 8)$-violation or an \emph{alternating 4-cycle} (a 4-cycle that is alternately colored using 2 colors). We avoid these alternating 4-cycles since their presence impedes the verification of the conditions of Theorem~\ref{thm:A} for the other conflicts.

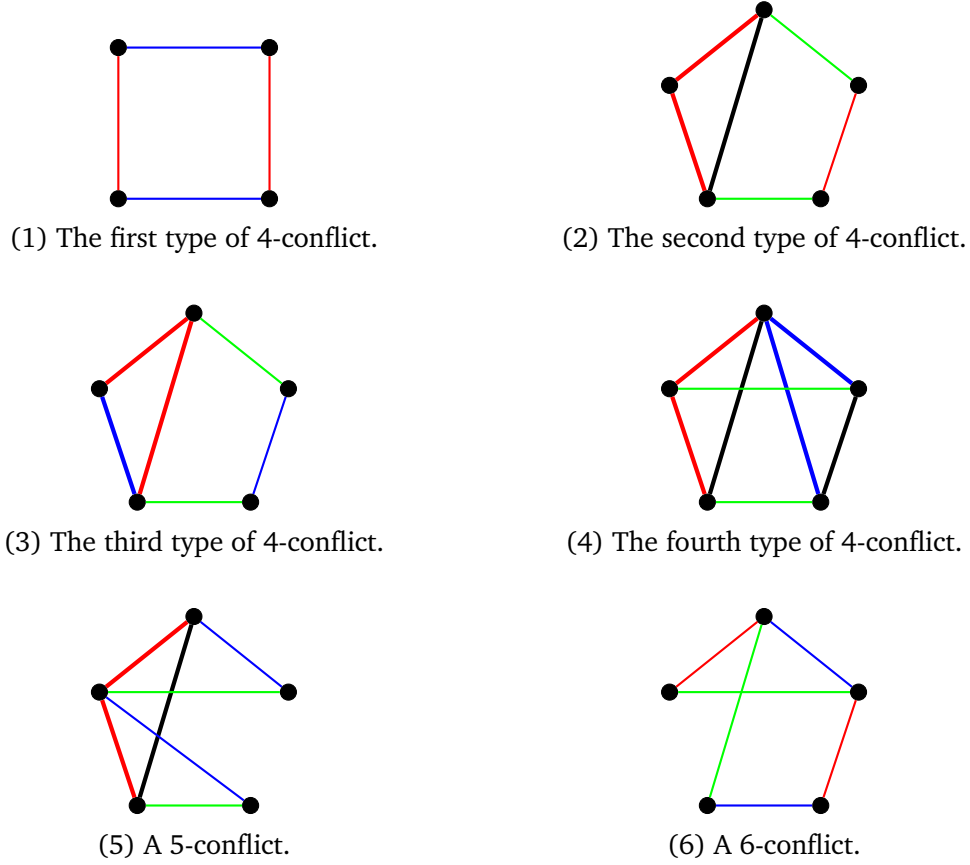
\begin{figure}[ht]
\begin{subfigure}[t]{.45\textwidth}
\centering
    \begin{tikzpicture}
	\node (u) at (0,0) [vert] {};
	\node (v) at (2,0) [vert] {};
	\node (x) at (0,2) [vert] {};
	\node (y) at (2,2) [vert] {};
	\draw [thick,red] (x) -- (u);
	\draw [thick,red] (y) -- (v);
    \draw [thick,blue] (u) -- (v);
    \draw [thick,blue] (y) -- (x);
\end{tikzpicture}
\subcaption{The first type of 4-conflict.}
\end{subfigure}
\vspace{3ex}
\begin{subfigure}[t]{.45\textwidth}
\centering
\begin{tikzpicture}[scale=.5]
\node (1) at (0,4) [vert] {};
\node (2) at (2.5,6) [vert] {};
\node (3) at (5,4) [vert] {};
\node (4) at (4,1) [vert] {};
\node (5) at (1,1) [vert] {};
\draw [red, ultra thick] (1) -- (2);
\draw [red, ultra thick] (1) -- (5);
\draw [black, ultra thick] (2) -- (5);
\draw [thick,green] (2) -- (3);
\draw [thick,green] (5) -- (4);
\draw [thick,red] (3) -- (4);
\end{tikzpicture}
\subcaption{The second type of 4-conflict.}
\end{subfigure}
\vspace{3ex}
\begin{subfigure}[t]{.45\textwidth}
\centering
\begin{tikzpicture}[scale=.5]
\node (1) at (0,4) [vert] {};
\node (2) at (2.5,6) [vert] {};
\node (3) at (5,4) [vert] {};
\node (4) at (4,1) [vert] {};
\node (5) at (1,1) [vert] {};
\draw [red, ultra thick] (1) -- (2);
\draw [red, ultra thick] (2) -- (5);
\draw [blue, ultra thick] (1) -- (5);
\draw [thick,green,] (2) -- (3);
\draw [thick,green] (5) -- (4);
\draw [thick,blue] (3) -- (4);
\end{tikzpicture}
\subcaption{The third type of 4-conflict.}
\end{subfigure}
\begin{subfigure}[t]{.45\textwidth}
\centering
\begin{tikzpicture}[scale=.5]
\node (1) at (0,4) [vert] {};
\node (2) at (2.5,6) [vert] {};
\node (3) at (5,4) [vert] {};
\node (4) at (4,1) [vert] {};
\node (5) at (1,1) [vert] {};
\draw [red, ultra thick] (1) -- (2);
\draw [black, ultra thick] (2) -- (5);
\draw [blue, ultra thick] (2) -- (3);
\draw [blue, ultra thick] (2) -- (4);
\draw [red, ultra thick] (1) -- (5);
\draw [black, ultra thick] (3) -- (4);
\draw [thick,green] (5) -- (4);
\draw [thick,green] (1) -- (3);
\end{tikzpicture}
\subcaption{The fourth type of 4-conflict.}
\end{subfigure}
\begin{subfigure}[t]{.45\textwidth}
\centering
\begin{tikzpicture}[scale=.5]
	\node (1) at (0,4) [vert] {};
\node (2) at (2.5,6) [vert] {};
\node (3) at (5,4) [vert] {};
\node (4) at (4,1) [vert] {};
\node (5) at (1,1) [vert] {};
\draw [red, ultra thick] (1) -- (2);
\draw [red, ultra thick] (1) -- (5);
\draw [black, ultra thick] (2) -- (5);
\draw [thick,green] (5) -- (4);
\draw [thick,green] (1) -- (3);
\draw [thick,blue] (1) -- (4);
\draw [thick,blue] (2) -- (3);
\end{tikzpicture}
\subcaption{A 5-conflict.}
\end{subfigure}
\begin{subfigure}[t]{.45\textwidth}
\centering
\begin{tikzpicture}[scale=.5]
\node (1) at (0,4) [vert] {};
\node (2) at (2.5,6) [vert] {};
\node (3) at (5,4) [vert] {};
\node (4) at (4,1) [vert] {};
\node (5) at (1,1) [vert] {};
\draw [thick,red] (1) -- (2);
\draw [thick,red] (3) -- (4);
\draw [thick,blue] (2) -- (3);
\draw [thick,blue] (4) -- (5);
\draw [thick,green] (1) -- (3);
\draw [thick,green] (2) -- (5);
\end{tikzpicture}
\subcaption{A 6-conflict.}
\end{subfigure}
\caption{Representative conflicts.  Bold edges come from the triangle of a gadget. Black is used for edges in triangles whose color does not repeat. Other than bold triangles, colored edges come from different gadgets.}
\label{fig:conflict}
\end{figure}

\begin{itemize}
\item {\bf Type 1 conflicts}  (see Figure~\ref{fig:conflict}(1)) correspond to collections of 4 gadgets, each having an edge in some set $S$ of 4 vertices, making an alternating 4-cycle in $S$. 

\item {\bf Type 2 conflicts} (see Figure~\ref{fig:conflict}(2)) correspond to collections of 4 gadgets, say $\Gamma_1, \ldots, \Gamma_4$ making a $(5, 8)$-violation $S$ in the following way. $\G_1$ has a triangle in $S$ with two $c$-colored edges (red in the figure). $\G_2$ also has a $c$-colored edge in $S$, disjoint from the other $c$-colored edges. $\G_3$ and $\G_4$ each have a $c'$-colored edge (green) in $S$ and these two edges are disjoint. Furthermore there is no alternating 4-cycle (in particular, no red-green cycle).

\item {\bf Type 3 conflicts}  (see Figure~\ref{fig:conflict}(3)) correspond to collections of 4 gadgets, say $\Gamma_1, \ldots, \Gamma_4$ making a $(5, 8)$-violation $S$ in the following way. $\Gamma_1$ has a triangle in $S$ with two $c$-colored edges (red) and one $c'$-colored edge (blue). $\Gamma_2$ has one edge in $S$ colored $c'$. $\Gamma_3, \Gamma_4$ each have one $c''$-colored edge (green) in $S$ and these edges are disjoint.  There is no alternating 4-cycle (in particular, no green-blue cycle).

\item {\bf Type 4 conflicts} (see Figure~\ref{fig:conflict}(4)) correspond to a collections of 4 gadgets $\Gamma_1, \ldots, \Gamma_4$ making a $(5, 8)$-violation $S$ in the following way. $\Gamma_1$ and $\Gamma_2$ each have a triangle in $S$. These two triangles share exactly one vertex and can be ``oriented'' in any way (the figure shows the orientation where the shared vertex has red degree 1 and blue degree 2). Then $\Gamma_3$ and $\Gamma_4$ each contribute one $c$-colored edge (green), where these two edges are disjoint. 

\item {\bf Type 5 conflicts} (see Figure~\ref{fig:conflict}(5)) correspond to collections of 5 gadgets $\Gamma_1, \ldots, \Gamma_5$ making a $(5, 8)$-violation $S$ in the following way. $\Gamma_1$ has a triangle  in $S$. $\G_2$ and $\G_3$ each have $c$-colored edge (blue) in $S$, and $\G_4, \G_5$ each have a $c'$-colored edge (green) in $S$. The $c$-colored edges are disjoint, as are the $c'$-colored edges, and there is no alternating 4-cycle in $S$. In particular the $c$- and $c'$-colored edges together make an alternately colored path of length 4 $(s_1, s_2, \ldots, s_5)$ where $s_1, s_3, s_5$ are all on the triangle from $\G_1$. This path could be oriented in a few ways with respect to the orientation of the triangle from $\G_1$, and the figure shows one possible orientation.

\item {\bf Type 6 conflicts} (see Figure~\ref{fig:conflict}(6)) correspond to collections of 6 gadgets making a $(5, 8)$-violation $S$ in the following way. Each gadget has an edge in $S$, and these 6 edges are colored with 3 colors, with 2 edges in each of the 3 colors $c, c', c''$. Edges sharing a color are disjoint, and there is no alternating 4-cycle. In particular the union of any two colors from $c, c', c''$ is an alternating 5-path. There are a few ways these edges can be arranged with respect to each other, and the figure shows one way. 

\end{itemize}

\begin{prop}\label{clm:validcoloring}
   A $\mc{C}_A$-free matching $\mc{M}$ in $\mc{H}_A$ corresponds to a partial $(5, 8)$-coloring of $K_n$. 
\end{prop}
\begin{proof}
Suppose to the contrary that $S$ is a $(5, 8)$-violation in a partial coloring corresponding to a $\mc{C}_A$-free matching $\mc{M}$ in $\mc{H}_A$. So $|S|=5$ and $S$ has 3 color repetitions. Note that by the construction of $\mc{H}_A$, each monochromatic component in our coloring is contained in a gadget. This is due to the fact that if a gadget (edge of $\mc{H}_A$) corresponds to a set of colored edges where some vertex $v$ is touching a color $c$, then formally this gadget includes the vertex $(v, c) \in V(\mc{H}_A)$ and in a matching no other gadget can contain $(v, c)$ (i.e.\ no other gadget can contribute another $c$-colored edge touching $v$).  Next, we make the following observation.
\begin{observation}\label{obs:1}
    No gadget can have 4 vertices in our $(5, 8)$-violation $S$. 
\end{observation}
\noindent Indeed,
suppose $\Gamma$ has $4$ vertices $p, q, r, s$ in $S$, contributing at most two repeats (see Figure \ref{fig:4vtxs}). No other gadget in $\mc{M}$ can assign a color to any edge in $\{p, q, r, s\}$ since these are all either sharable edges or already colored by $\Gamma$. Letting $t$ be the fifth vertex in $S$, note that there cannot be an edge $tx$ for $x\in\{p, q, r, s\}$ of a color $c$ that appears in $\Gamma$. Indeed, either $c$ is non-touching for $x$ or is already touching $x$ via $\Gamma$. There also cannot be a gadget with a triangle in $S$ containing $t$ since the other two vertices of this triangle would be in $\{p, q, r, s\}$. Thus there is no way to get a third repeat to make $S$ a $(5, 8)$-violation. 
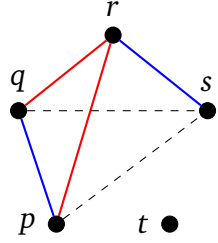
\begin{figure}[ht]
\centering
\begin{tikzpicture}[scale=.5]
\node   [label=$q$](1) at (0,4) [vert] {};
\node  [label=$r$] (2) at (2.5,6) [vert] {};
\node [label=$s$] (3) at (5,4) [vert] {};
\node  [label=left:$t$](4) at (4,1) [vert] {};
\node   [label=left:$p$](5) at (1,1) [vert] {};
\draw [red] (1) -- (2);
\draw [red] (2) -- (5);
\draw [dashed] (5) -- (3);
\draw [dashed] (1) -- (3);
\draw [blue] (1) -- (5);
\draw [blue] (2) -- (3);
\end{tikzpicture}
\caption{A gadget with 4 vertices in $S$.}
\label{fig:4vtxs}
\end{figure}

We make another observation:

\begin{observation}\label{obs:2}
    Suppose a gadget $\Gamma$ has 3 vertices in $S$ spanning one edge of color $c$, one edge of color $c'$ and one sharable edge. Then at most one of $c$ and $c'$ can be repeated in $S$. 
\end{observation}
\noindent Indeed, by the previous observation $\Gamma$ cannot have a fourth vertex in $S$. For the 3 vertices $\Gamma$ has in $S$, both $c, c'$ are either non-touching for these vertices (as is the case for red with $s$ in Figure~\ref{fig:4vtxs}) or else they are already touching due to $\Gamma$ (as is the case for red with $p, r, q$ and for blue with $p,q,r,s$ in Figure~\ref{fig:4vtxs}). Thus $S$ only has one edge which could be colored either $c$ or $c'$. 

 Together, Observations \ref{obs:1} and \ref{obs:2} imply the following:
 \begin{observation}\label{obs:3}
     The only way for a gadget $\Gamma$ to contribute to color repetitions in $S$ is for $S$ to either contain exactly 3 vertices of $\Gamma$ forming a triangle in $\Gamma$, or else for $\Gamma$ to have exactly one edge of a repeated color in $S$.
 \end{observation}

We consider the following cases for how to get 3 color repetitions. 

{\bf Case 1:} $S$ contains 4 edges of the same color. This is not possible since all monochromatic components in our coloring have at most 2 edges. 

{\bf Case 2:} $S$ contains 3 $c_1$-colored edges and 2 $c_2$-colored edges. The $c_1$-colored components induced in $S$ must consist of a 2-path $pqr$ and an isolated edge $st$ (see Figure~\ref{fig:3reps2cols} (1)). Let $\Gamma_1$ be the gadget with the $c_1$-colored 2-path. 
\begin{figure}[t]
\begin{subfigure}{.33\textwidth}
\centering
\begin{tikzpicture}[scale=.5]
\node[label=$q$] (1) at (0,4) [vert] {};
\node[label=above:$r$] (2) at (2.5,6) [vert] {};
\node[label=$s$] (3) at (5,4) [vert] {};
\node[label=left:$t$] (4) at (4,1) [vert] {};
\node[label=left:$p$] (5) at (1,1) [vert] {};
\draw [red] (1) -- (2);
\draw [red] (1) -- (5);
\draw [blue] (2) -- (5);
\draw [red] (3) -- (4);
\end{tikzpicture}
\caption{Two repetitions in the color $c_2$.}
\end{subfigure}
\begin{subfigure}{.66\textwidth}
\centering
\begin{tikzpicture}[scale=.5]
\node (1) at (0,4) [vert] {};
\node (2) at (2.5,6) [vert] {};
\node (3) at (5,4) [vert] {};
\node (4) at (4,1) [vert] {};
\node (5) at (1,1) [vert] {};
\draw [red] (1) -- (2);
\draw [red] (1) -- (5);
\draw [blue] (2) -- (5);
\draw [red] (3) -- (4);
\draw [green] (4) -- (5);
\draw [green] (3) -- (2);
\end{tikzpicture}
\begin{tikzpicture}[scale=.5]
\node (1) at (0,4) [vert] {};
\node (2) at (2.5,6) [vert] {};
\node (3) at (5,4) [vert] {};
\node (4) at (4,1) [vert] {};
\node (5) at (1,1) [vert] {};
\draw [red] (1) -- (2);
\draw [red] (1) -- (5);
\draw [blue] (2) -- (5);
\draw [red] (3) -- (4);
\draw [green] (4) -- (1);
\draw [green] (3) -- (2);
\end{tikzpicture}
\begin{tikzpicture}[scale=.5]
\node (1) at (0,4) [vert] {};
\node (2) at (2.5,6) [vert] {};
\node (3) at (5,4) [vert] {};
\node (4) at (4,1) [vert] {};
\node (5) at (1,1) [vert] {};
\draw [red] (1) -- (2);
\draw [red] (1) -- (5);
\draw [blue] (2) -- (5);
\draw [red] (3) -- (4);
\draw [green] (3) -- (1);
\draw [green] (4) -- (5);
\end{tikzpicture}
\caption{Three repetitions in the colors $c_2,c_3$.}
\end{subfigure}
\caption{Configurations involving three edges of the same color.}
\label{fig:3reps2cols}
\end{figure}
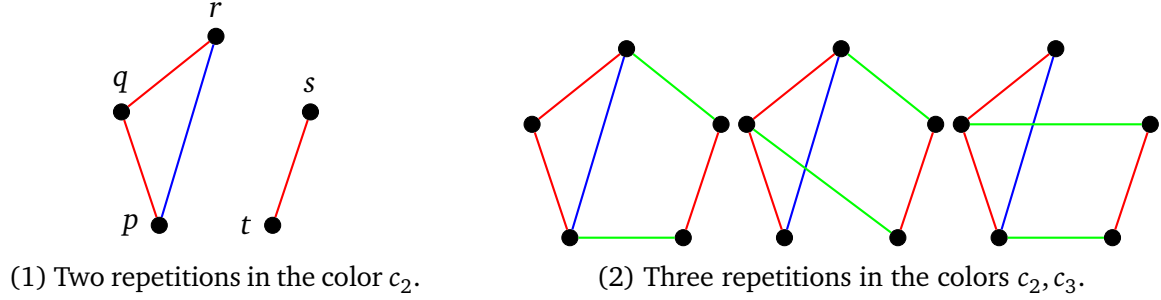
We claim that $pr$ cannot be colored $c_2$ (i.e.\ it cannot be that $c_2$ is blue and $c_1$ is red in  Figure~\ref{fig:3reps2cols} (1)). Indeed, in that case there must be another $c_2$-colored edge $e$ in $S$. It is not possible that $e$ is part of $\Gamma_1$ by Observation \ref{obs:3}. $e$ could also not be part of a different gadget since it would have to have one end at $p, q$ or $r$ which are all already touching $c_2$. Thus, $c_2$ cannot be the color blue in  Figure~\ref{fig:3reps2cols} (1). So there must be two other edges in $S$ colored $c_2$, in one of the configurations in Figure~\ref{fig:3reps2cols} (2). Note that these two $c_2$-colored edges cannot come from the same gadget by Observation \ref{obs:3}. No matter how we arrange the edges, we either end up with a Type 1 or 2 conflict (see Figure~\ref{fig:conflict}(1) and (2)). Thus, this case is impossible. 

{\bf Case 3:} $S$ has 2 edges in each of the colors $c_1, c_2, c_3$. By Observation \ref{obs:3}, there are two possibilities for each color $c_i$: either there is a gadget contributing a $c_i$-colored $2$-path to $S$, or two different gadgets each contributing a $c_i$-colored edge to $S$ and these two edges are nonadjacent. Note that we cannot fit 3 edge-disjoint triangles in $S$, and so it cannot be the case that we have a 2-path in each color $c_1, c_2$ and $c_3$. If we have 2-paths in 2 of our colors, say $c_1$ and $c_2$ then we have a Type 4 conflict (see Figure~\ref{fig:conflict}(4)). If we have a 2-path in only one color we have a Type 1, 3, or 5 conflict (see Figure~\ref{fig:conflict}(1),(3),(5)), and if we have no 2-paths we have a Type 1 or 6 conflict (see Figure~\ref{fig:conflict}(1),(6)). These are all contradictions since $\mc{M}$ is $\mc{C}_A$-free. Thus we have a partial $(5, 8)$-coloring. 
\end{proof}

In light of Proposition \ref{clm:validcoloring} and the fact that $C_A$ has $\frac 67 n+ O(n^{1-\d})$ colors, our proof of Theorem~\ref{thm:phaseA} will be done if we show there is a large enough $\mc{C}_A$-free matching in $\mc{H}_A$. To do that, we start checking the conditions of Theorem \ref{thm:A}. 

\subsection{Checking conditions of Theorem \ref{thm:A}}
Recall that we will use $k_A=19$. $\ell_A, \eps_A$ will be determined later. Set
\[
    d_A:= \frac{27}{196}n^7p^{14} + n^7p^{14.5}.
\]
$|V(\mc{H})|$ is polynomial in $n$ so it will be less than $\exp(d_A^{\eps_A^3})$. 
\begin{claim}
     $\Delta_2(\mc{H}_A) =O(n^6)$. Therefore \ref{cond:h1a} holds. 
\end{claim}

\begin{proof}
     We fix two vertices in $\mc{H}_A$ and bound the number of edges of $\mc{H}_A$ (i.e.\ colored gadgets $\Gamma$) containing them. Recall that $V(\mc{H}_A)=E' \cup V'$ (see \eqref{eqn:VHA}) and so we consider the following cases for our two fixed vertices of $\mc{H}_A$.
     
     If we fix $e, e' \in E'$ then since $e, e'$ must be in $\Gamma$ we choose at most $3$ more vertices from $K_n$ and 3 colors so there are $O(n^6)$ edges of $\mc{H}_A$ containing $e, e'$. 

     If we fix $e \in E'$, $(v, c) \in V'$ then we have fixed at least 2 vertices of $\Gamma$ and one color. Thus we choose at most 4 vertices and 2 colors and so there are $O(n^6)$ ways to complete an edge of $\mc{H}_A$. 

     If we fix distinct $(v, c), (v', c') \in V'$ then we have either fixed two vertices and at least one color, or two colors and at least one vertex. Either way there are $O(n^6)$ ways to complete an edge of $\mc{H}_A$. Thus the claim follows assuming that $14\d + 7\eps_A <1$ which follows from \eqref{eqn:ed}.
\end{proof}

\begin{claim}\label{claim:H2A}
    $(1-d_A^{-\eps_A})d_A\leq \delta(\mc{H}_A)\leq\Delta(\mc{H}_A)\leq d_A$. Therefore \ref{cond:h2a} holds. 
\end{claim}

\begin{proof}
We now find the expected degree of a vertex in $\mc{H}_A$. Recall that a vertex in $\mc{H}_A$ is either an edge $e \in E' \subseteq E(K_n)$ or a vertex-color pair $(v, c) \in V'$, where $c\in C_{A1}$ or $c\in C_{A2}$. First consider $e \in E'$. There are three distinct ways that $e$ could play in $\Gamma$. The first case is when $e=\gamma_i\gamma_j$ where $ij = 12, 56$ (see Figure~\ref{fig:G}). The expected number of edges of $\mc{H}_A$ using $e$ this way is 
\begin{align}
 & (n-2)(n-3) \binom{n-4}{2}|C_{A1}| |C_{A2}|^2 p^{\binom{6}{2}-7}p^{6} (1-p)^6 (1-p)^{6} \nn\\
& =  \frac 12 n^4 |C_{A1}| |C_{A2}|^2p^{14} \pm O(n^7p^{15}).\label{eqn:deg1}
\end{align}
On the other hand,  $e$ could also play the role of $\gamma_i\gamma_j$ in a $\Gamma$, where $ij = 13, 23, 45, 46$. The expected number of edges of $\mc{H}_A$ using $e$ this way is  
\begin{equation}\label{eqn:deg2}
  n^4 |C_{A1}| |C_{A2}|^2p^{14}\pm O(n^7p^{15}).
\end{equation}
The final way that $e$ could be a part of $\Gamma$ is for $ij=34$. The expected number of edges of $\mc{H}_A$ using $e$ this way is
\begin{equation}\label{eqn:deg3}
 \frac 14 n^4 |C_{A1}| |C_{A2}|^2 p^{14} \pm O(n^7p^{15}).
\end{equation}
Therefore, summing \eqref{eqn:deg1}, \eqref{eqn:deg2} and \eqref{eqn:deg3} we have
\begin{align*}
\mathbb{E} [\deg_{\mc{H}_A}(e) \mid e \in E'] &= \left(\frac12+1+\frac 14\right)n^4 |C_{A1}||C_{A2}|^2 p^{14}\pm O(n^7p^{15})\nn\\
&= \frac{7}{4}\left(\frac{3}{7}\right)^3 n^7 p^{14} \pm O(n^7p^{15})\nn\\
&= \frac{27}{196} n^7 p^{14}\pm O(n^7p^{15}).\label{eqn:d1}
\end{align*}

Now consider a vertex of $\mc{H}$ of the form $(v, c)$ where $c\in C_{A1}$. First we find the expected number of edges of $\mc{H}$ where $v$ plays the role of $\gamma_i$ for $i=1,2,5,6$:
\begin{equation}\label{eqn:deg4}
 \frac12 n^5  |C_{A2}|^2 p^{14}  \pm O(n^7p^{15}).
\end{equation}
For $c \in C_{A1}$ and $i=3,4$ the expected number of edges of $\mc{H}$ is
\begin{equation}\label{eqn:deg5}
  \frac14 n^5  |C_{A2}|^2p^{14}\pm O(n^7p^{15}) .
\end{equation}

So, summing \eqref{eqn:deg4} and \eqref{eqn:deg5}, we have for $c \in C_{A1}$ that
\begin{align}
\mathbb{E} [\deg_{\mc{H}_A}(v,c) \mid (v,c) \in V'] &= \left(\frac 12 + \frac 14\right)\left(\frac{3}{7}\right)^2n^7p^{14}  \pm O(n^7p^{15}) \nn\\
&= \frac{27}{196}n^7p^{14}  \pm O(n^7p^{15}).\label{eqn:d2}
\end{align}
Finally, for $c\in C_{A2}$,
\begin{align*}
\mathbb{E} [\deg_{\mc{H}_A}(v,c) \mid (v,c) \in V'] &=  n^3 \binom{n}{2}|C_{A1}||C_{A2}| p^{14}(1-p)^{12}\\
&\qquad\qquad + n\binom{n}{2}^2|C_{A1}||C_{A2}|p^{14} (1-p)^{12} \pm O(n^6)\nn\\
&=\frac{27}{196}n^7p^{14} \pm O(n^7p^{15}).\label{eqn:d3}
\end{align*}

We use McDiarmid's inequality (Theorem \ref{thm:mcd}) to show concentration of $\deg_{\mc{H}_A}(e)$ around its expected value. Define $X_{f}$ to be the indicator function for $f$ in $E'$ and $Y_{(v, c)}$ to be the indicator function for $(v, c)$ in $V'$. Then for $f$ in $E$, define $b_f$ to be $n^{3+3}$ when $f$ is incident to $e$ and $n^{2+3}$ otherwise. For $(v, c)$ in $V$, let $b_{(v, c)}$ be $n^{4+2}$ when $v$ is in $e$ and $n^{2+3}$ otherwise. Observe that in any of these cases $\deg_{\mc{H}_A}(e)$ changes by at most $b_i$, where $i=f$ or $(v, c)$. Then notice that $\sum_i b_i^2 =  O\left(n \cdot (n^{3+3})^2\right ) + O\left(n^2 \cdot (n^{2+3})^2\right ) + O\left(n \cdot (n^{4+2})^2\right ) + O\left(n^2 \cdot (n^{3+2})^2\right )=O(n^{13})$. Thus, McDiarmid's inequality with $t=\frac12 n^7p^{14.5}$ implies that w.h.p. $$\frac{27}{196}n^7p^{14}  -n^7p^{14.5} \le \deg_{\mc{H}_A}(e) \le \frac{27}{196}n^7p^{14}  + n^7p^{14.5}$$ for all $e$ (the failure probability in McDiarmid's inequality is small enough for the union bound over $e$).

We follow a similar argument for $\deg_{\mc{H}_A}((v, c))$. The calculations are equal since we are given one less vertex but one more color.  Therefore, w.h.p. $\mc{H}_A$ satisfies 
$$\frac{27}{196}n^7p^{14} -n^7p^{14.5} \le \delta(\mc{H}_A) \le \Delta(\mc{H}_A) \le \frac{27}{196}n^7p^{14} + n^7p^{14.5}.$$
The above upper bound is exactly $d_A$, and assuming that $\eps_A < \d / (14-28\d)$ which follows from \eqref{eqn:ed}, the lower bound above is at least $(1-d_A^{-\eps_A})d_A$.
\end{proof}

\begin{figure}[ht]
\begin{center}
 \begin{subfigure}{.48\textwidth}
\centering
\begin{tikzpicture}[scale=1]
\begin{scope}[rotate=0, shift={(0,0)}]
	\DrawGamma{blue}{black}{green}{$\Gamma_4$}
\end{scope}
\draw[blue, line width=2pt] (g1) -- (g2);
\begin{scope}[shift={(-1.87,1.87)}, rotate=90]
	\DrawGamma{red}{orange}{gray}{$\Gamma_3$}
\end{scope}
\draw[red, line width=2pt] (g1) -- (g2);
\begin{scope}[shift={(-3.74,0)}, rotate=0]
	\DrawGamma{blue}{teal}{violet}{$\Gamma_2$}
\end{scope}
\draw[blue, line width=2pt] (g5) -- (g6);
\begin{scope}[shift={(-1.87,-1.87)}, rotate=270]
	\DrawGamma{red}{yellow}{cyan}{$\Gamma_1$}
\end{scope}
\draw[red, line width=2pt] (g1) -- (g2);
\end{tikzpicture}
\caption{}
\end{subfigure}
\hfill
\begin{subfigure}{.48\textwidth}
\centering
\begin{tikzpicture}[scale=1]
\begin{scope}[rotate=-90, shift={(0,0)}]
	\DrawGamma{blue}{black}{black}{$\Gamma_0$}
\end{scope}
\draw[blue, line width=2pt] (g1) -- (g2);
\begin{scope}[shift={(-1.955,1.425)}, rotate=-90-72]
	\DrawGamma{green}{black}{black}{$\Gamma_1$}
\end{scope}
\draw[green, line width=2pt] (g1) -- (g2);
\begin{scope}[shift={(-1.21,3.73)}, rotate=-90-72-72]
	\DrawGamma{red}{black}{black}{$\Gamma_2$}
\end{scope}
\draw[red, line width=2pt] (g1) -- (g2);
\begin{scope}[shift={(1.21,3.73)}, rotate=-90-72-72-72]
	\DrawGamma{blue}{black}{black}{$\Gamma_3$}
\end{scope}
\draw[blue, line width=2pt] (g1) -- (g2);
\begin{scope}[shift={(1.955,1.425)}, rotate=-90-72-72-72-72]
	\DrawGamma{red}{black}{black}{$\Gamma_4$}
\end{scope}
\draw[red, line width=2pt] (g1) -- (g2);
\begin{scope}[shift={(2.36,2.83)}, rotate=-90-72-180, scale=1.618]
	\DrawGamma{green}{black}{black}{$\Gamma_5$}
\end{scope}
\draw[green, line width=2pt] (g1) -- (g2);
\end{tikzpicture}
\caption{}
\end{subfigure}
\caption{A 4-conflict (left) and a 6-conflict (right). The black edges are not given colors to aid readability.}
\label{fig:C1-C3}
\end{center}
\end{figure}
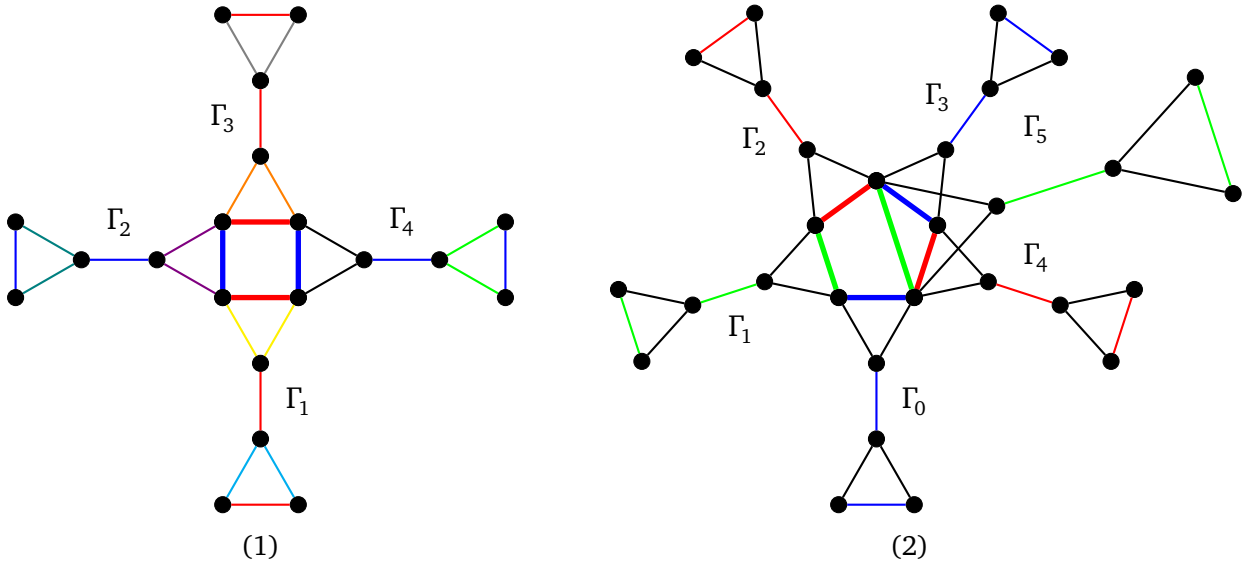

We move on to checking the conditions in Theorem \ref{thm:A} relevant to the conflict system. Recall that the conflicts in $\mc{C}_A$ have size at most 6, so $\mc{C}_A$ is $\ell_A$-bounded as long as $\ell_A \ge 6$. 

\begin{claim}
    \ref{cond:c1a} holds for some constant $\ell_A$.
\end{claim}
\begin{proof}
We start verifying the claim with $j=4$. We will check one case in more detail and then provide a general argument for the rest. We will check Type 1 conflicts in detail, i.e.\ we will check that each gadget is contained in $O(d^3)$ conflicts of Type 1. Fix a gadget $\G_1$. A Type 1 conflict in $\mc{C}_A^{(4)}$ containing $\G_1$ contains 3 other gadgets $\Gamma_2, \Gamma_3, \Gamma_4$ (see Figure~\ref{fig:C1-C3}(1)). Fixing $\Gamma_1$ and assuming we choose all the information for $\G_2, \G_3, \G_4$ in that order, the choices we must make are summarized in Table~\ref{tab:C1A}. Note that this involves choosing several vertices and colors, but certain pairs of the vertices we choose must not induce edges in $E'$ and some pairs of vertices and colors must not be in $V'$. The table also summarizes these requirements.
\begin{table}[h!]
\begin{center}
\begin{tabular}{c|ccc|c}
& $\Gamma_2$ & $\Gamma_3$ & $\Gamma_4$ & total\\
\hline
vertices & 5 & 5 & 4 & 14\\
colors & 3 & 2 & 2 & 7\\
$e \not\in E'$ & 8 & 8 & 8 & 24\\
$(v, c) \not\in V'$ & 6 & 6 & 6 & 18\\
\end{tabular}
\end{center}
\caption{Counting vertices, colors, sharable edges and non-touching colors.}
\label{tab:C1A}
\end{table}
Each vertex and each color has $O(n)$ choices. The event that any particular pair of vertices induces an edge that is not in $E'$ is $p$, which is also the probability of any particular pair $(v, c)$ not being in $V'$. Thus, the expected number of ways to complete our Type 1 conflict is $O(n^{14+7}p^{24+18})=O(d^3)$. In a similar approach to that outlined in the proof of Claim~\ref{claim:H2A}, McDiarmid's inequality shows concentration and so w.h.p.~each gadget $\G_1$ is in $O(d^3)$ conflicts of Type 1. Notice that we needed to include factors of $p$ in our estimate. Oftentimes, such as in Claim~\ref{claim:C2A}, we do not need to be as precise in order to verify the required condition. 

Now we claim that each gadget is in $O(d^3)$ conflicts of Types 2, 3, and 4 (i.e.\ the rest of the conflicts in $\mc{C}_A^{(4)}$). These conflicts all involve 4 gadgets such that the sum of the number of vertex choices and color choices is 30. Indeed, for Type 2 and 3 conflicts there are 20 vertices spanned and 10 colors used, while for Type 4 conflicts there are 19 vertices spanned and 11 colors used. Any of these conflicts involves having 32 pairs of vertices not in $E'$ and 24 vertex-color pairs not in $V'$. Fixing one of the gadgets gives us 6 vertices, 3 colors, 8 pairs not in $E'$ and 6 pairs not in $V'$. Thus the expected number of conflicts of Type 2, 3, or 4 containing a fixed gadget is $O(n^{30-6-3}p^{32+24 - 8 - 6}) = O(d^3)$. As before, McDiarmid's inequality shows concentration.

We move on to $\mc{C}_A^{(5)}$, meaning Type 5 conflicts. We claim that each gadget is in $O(d^4)$ conflicts of Type 5. Indeed, such a conflict involves 5 gadgets spanning a total of 24 vertices, using a total of 13 colors, having 40 pairs of vertices not in $E'$ and 30 vertex-color pairs not in $V'$. Thus the expected number of conflicts of Type 5 containing a fixed gadget is $O(n^{24+13-6-3}p^{40+30 - 8 - 6}) = O(d^4)$.

Finally we consider $\mc{C}_A^{(6)}$, meaning Type 6 conflicts. These conflicts involve 6 gadgets spanning 29 vertices, using 15 colors, having 48 pairs not in $E'$ and 36 pairs not in $V'$.  Thus the expected number of conflicts of Type 5 containing a fixed gadget is $O(n^{29+15-6-3}p^{48+36 - 8 - 6}) = O(d^5)$.

For each $j=4, 5, 6$ we have shown that $\Delta(\mc{C}_A^{(j)})=O(d^{j-1})$. Thus, there exists an $\ell'$ such that $\Delta(\mc{C}_A^{(j)}) \le \ell' d^{j-1}$ for each $j$. 
\[
    \ell_A:=\max(\ell', 6),
\]
then \ref{cond:c1a} is verified. 
\end{proof}

\begin{claim}\label{claim:C2A}
    \ref{cond:c2a} holds. 
\end{claim}
\begin{proof}
As with the previous claim, we will check Type 1 conflicts in more detail and provide a general argument for the rest of the types. For Type 1 conflicts we have $j=4$ so $j'=2, 3$. For $j'=2$ so we fix some $\G_i, \G_j$. There are two cases: $\Gamma_i$ and $\Gamma_j$ either share a vertex or a color. If they share a vertex, then we need to choose 9 more vertices and 4 more colors. If they share a color, then we need 8 more vertices and 5 more colors. Thus, the number of Type 1 conflicts containing $\G_i, \G_j$ is $O(n^{13})$ which is smaller than $d^{4-2-\eps_A}$ by \eqref{eqn:ed} (the important part here is that $d$ has a factor $n^{7}$).
 Now for $j'=3$, there is only one unfixed $\Gamma_i$ for which we need to choose 4 vertices and 2 colors. So the number of Type 1 conflicts containing 3 fixed gadgets is $O(n^{6}) < d^{4-3-\eps_A}$. This completes our detailed discussion of Type 1 conflicts.

In general, consider a conflict of any of the types in $\mc{C}_A^{(j)}$ for $j=4, 5, 6$. As we discussed in the proof of the previous claim, each such conflict involves a total of $7j+2$ vertices and colors. Now fix $j'$ gadgets in the conflict, for $2 \le j' < j$. The crucial observation (which the reader can check by looking at Figure \ref{fig:conflict}) is that for all of our conflict types, fixing any set of $j'$ gadgets fixes at least $7j'+3$ vertices and colors. This leaves $7(j-j')-1$ vertices and colors to choose, so our number of choices is $O(n^{7(j-j')-1}) < d^{j-j'-\eps_A}$ which follows from \eqref{eqn:ed} and the fact that $d$ has the factor $n^7$. So \ref{cond:c2a} is verified.
\end{proof}

Thus, all conditions of Theorem \ref{thm:A} are met, and so $\mc{H}_A$ has a $\mc{C}_A$-free matching of size asymptotically (recalling \eqref{eqn:nadef} and $k_A=19$) $n_A/k_A \sim n^2/14$. Since every edge of our matching colors 7 edges in our partial $(5, 8)$-coloring of $K_n$, the number of colored edges is asymptotically $n^2/2$. This completes the proof of  Theorem \ref{thm:phaseA}.
\end{proof}

As we discussed in Section \ref{sec:overview}, we have some leftover uncolored edges after Phase A (this includes both the sharable edges and some more edges). We plan to color them in Phase B by applying Theorem \ref{thm:B} to a new hypergraph $\mc{H}_B$ and conflict system $\mc{C}_B$ which together will encode the coloring problem we have for the remaining uncolored edges. However, for the purposes of checking that $\mc{H}_B, \mc{C}_B$ satisfy the conditions of Theorem \ref{thm:B} we must establish some quasirandom properties of the coloring produced in Phase A. The following lemma summarizes these properties. Recall that $E''$ is the set of sharable edges. Let $E'''$ be the set of non-sharable edges which do not get colored in Phase A. We say the edges in $E'' \cup E'''$ are the {\em uncolored edges}, and the {\em uncolored degree} of a vertex is the number of incident uncolored edges. 
\begin{definition}\label{def:Sab}
    For $e \in E'' \cup E'''$, let $\mc{S}_{a, b}(e)$ be the family of sets $S \subseteq V(K_n)$ with $|S|=a$, having $b$ color repetitions in the coloring from Phase A, and containing both $e$ and an uncolored edge $e'$ which is disjoint from $e$ and which belongs to the same set (either $E''$ or $E'''$) as $e$.
\end{definition} 
Roughly speaking, these sets $S$ represent potential danger for Phase B in the sense that if $e$ gets the same color as $e'$ in Phase B, then $S$ gains a repeat. We focus on the case where $e$ and $e'$ are disjoint since we will be able to do Phase B where every color class is a matching. Also, in Phase B we will use disjoint sets of colors for $E''$ and $E'''$ which is why we restrict our attention to the case when $e$ and $e'$ are in the same set. Note that $\mc{S}_{4, b}(e)=\emptyset$ for $b \ge 2$. Indeed, the only way for a set $S$ of 4 vertices to get 2 repeats in Phase A  is if some gadget has 4 vertices ( like say $\g_1, \g_2, \g_3, \g_4$ in Figure \ref{fig:G}) in $S$. But in that case, the only two uncolored edges in $S$ share an endpoint, so $S \notin \mc{S}_{4, 2}(e)$. Also, $\mc{S}_{5, b}(e)=\emptyset$ for $b \ge 3$ since Phase A  gives a partial $(5, 8)$-coloring.

\begin{lemma}\label{lem:quasi-random}
Our matching $\mc M_A$ and corresponding partial $(5,8)$-coloring from Phase A can be taken to satisfy the following:
\begin{enumerate}
\item  each vertex $v$ has uncolored degree $O(n^{1-\eps_A^3})$; \label{lem:quasi-random-deg}
\item for each shareable edge $uv \in E''$, the number of gadgets containing $uv$ as in Figure~\ref{fig:scaff} is $O(n^\delta)$ \label{lem:quasi-random-shareable}
\item for each uncolored edge $uv$ and every $(a, b) \in \{(4, 0), (4, 1), (5,0), (5, 1), (5, 2)\}$, we have $|\mc{S}_{a, b}(uv)|=O(n^{a-b-2})$ \label{lem:quasi-random-a-b}
\item  For each pair $e_1, e_2$ of uncolored edges, there are at most $O(n^{1-\delta/2})$ pairs of sets $S_1, S_2$ with the following properties. $S_j \in \mc{S}_{4,1}(e_j) \cup \mc{S}_{5,2}(e_j)$ for $j=1,2$. Furthermore there is a single edge $e' \subseteq S_1 \cap S_2$ which is playing the role of the second uncolored edge (besides $e_1$) in $S_1$ and (besides $e_2$) in $S_2$. \label{lem:quasi-random-x-lem}
\end{enumerate}
\end{lemma}

Much of the content of Lemma \ref{lem:quasi-random} is heuristically easy to see (and much of it is easy to prove). The proof of lemma \ref{lem:quasi-random} is in Appendix \ref{sec:quasirandom}, and here we will just explain why it is believable. We already know Phase A colors almost all edges so part (\ref{lem:quasi-random-deg}) makes sense. About $n^2/2$ edges are colored and about $n^{2-\d}/2$ are sharable and so part (\ref{lem:quasi-random-shareable}) makes sense. To heuristically see part (\ref{lem:quasi-random-a-b}), first note that to extend $uv$ to a set $S \in \mc{S}_{a, b}(uv)$ we need to choose $a-2$ vertices. Then recall that the Phase A coloring is the result of a random process in which colored gadgets are chosen one at a time, and heuristically the probability that one randomly chosen color matches another one (i.e.\ the probability of each color repeat) is on the order $1/n$.  Part (\ref{lem:quasi-random-x-lem}) is, as we will see in Appendix \ref{sec:quasirandom}, bounding a family of configurations of gadgets each of which is heuristically (and actually) unlikely to appear a lot.  

\begin{figure}[ht]
\centering
\begin{tikzpicture}[scale=1.5]
\begin{scope}[rotate=-18, shift={(0,0)}]
	\DrawGamma{blue}{red}{green}{}
\end{scope}
\node[label=below:$u$] (u) at (g2) {};
\node[label=below:$v$] (v) at (g4) {};
\draw[dashed] (u) -- (v);
\end{tikzpicture}
\caption{Structures bounded by Lemma~\ref{lem:quasi-random}(\ref{lem:quasi-random-shareable}).}
\label{fig:scaff}
\end{figure}
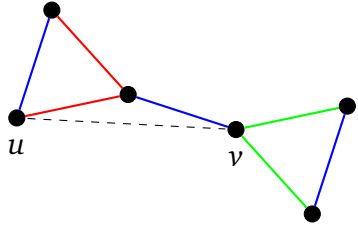
 
\section{Phase B}

We now describe $\mc{H}_B$ and $\mc{C}_B$ for our application of Theorem \ref{thm:B}. As we go we will verify some conditions of the theorem. $\mc{H}_B$ will be a bipartite graph (so $k_B=2$) with vertex set $X \cup Y$ where $X=E'' \cup E'''$ is the set of edges in $K_n$ which have not received a color by the end of Phase A. Set
\[
d_B:= n \log^{-1/10} n.
\]  
Let $C_B := C_{B1} \cup C_{B2}$  where $C_{B1}, C_{B2}$ are two disjoint sets of  $2 d_B$ new colors each. 
$Y$ will contain a vertex $e_c$ for every $(e, c) \in E'' \times C_{B1}$ and for every $(e, c) \in E''' \times C_{B2}$. Each $e \in X = E'' \cup E'''$ will be adjacent to every $e_c \in Y$. For clarity, $\mc{H}_B$ is just the disjoint union of stars, where each $e$ is the center of a star which has the $e_c$ as leaves. For $e \in X$ and $e_c \in Y$, we abuse notation by writing $e_c$ to also denote the edge $ee_c$ of $\mc{H}_B$. Condition \ref{cond:h1b} is satisfied since $\Delta_2(\mc{H}_B)=1$. Condition \ref{cond:h2b} is satisfied since every vertex in $X$ has degree $2d_B$ and every vertex in $Y$ has degree 1. 
An $X$-perfect matching in $\mc{H}_B$ corresponds to a coloring of the uncolored edges in the obvious way. Note that the sharable edges (in $E''$) only get colors from $C_{B1}$ and the nonsharable edges (in $E'''$) get colors from $C_{B2}$. 

Now we will define our conflict system $\mc{C}_B$ so that a $\mc{C}_B$-free matching in $\mc{H}_B$, together with our partial coloring from Phase A, gives us our final $(5, 8)$-coloring. Once we define $\mc{C}_B$, verifying Conditions \ref{cond:c1b}--\ref{cond:c4b} will take some more work.

To define $\mc{C}_B$ we will first define another conflict system $\mc{D}$ and then let $\mc{C}_B$ be the set of minimal (inclusion-wise) conflicts of $\mc{D}$. $\mc{D}$ will have the following conflicts, described in order of increasing size. 

{\bf 2-conflicts in $\mc{D}$:} $\{e_c, f_c\}$ will be a 2-conflict if any of the following holds:
\begin{itemize}
    \item $e$ and $f$ are adjacent in $K_n$,
    \item $f$ is in some $S \in \mc{S}_{4, 1}(e)$ (see Figure \ref{fig:Dconflicts}(1)), or 
    \item $f$ is in some $S \in \mc{S}_{5, 2}(e)$.
\end{itemize}
Note that when we write $e_c$ and $f_c$ this is implying that $e$ and $f$ belong to the same set $E''$ or $E'''$ (or else they could not get the same subscript). 

{\bf 4-conflicts in $\mc{D}$:} $\{e_c, f_c, e'_{c'}, f'_{c'} \}$ will be a 4-conflict if any of the following holds:
\begin{itemize}
    \item $f, e', f'$ are in some $S \in \mc{S}_{4, 0}(e)$ (see Figure \ref{fig:Dconflicts}(2)), or 
    \item $f, e', f'$ are in some $S \in \mc{S}_{5, 1}(e)$ (see Figure \ref{fig:Dconflicts}(3)).
\end{itemize}

{\bf 6-conflicts in $\mc{D}$:} $\{e_c, f_c, e'_{c'}, f'_{c'}, e''_{c''}, f''_{c''} \}$ will be a 6-conflict if $f, e', f', e'', f''$ are all in some $S \in \mc{S}_{5, 0}(e)$.

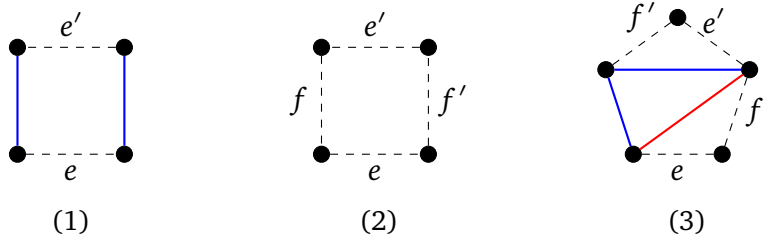
\begin{figure}[ht]
\begin{subfigure}{.24\textwidth}
\centering
\begin{tikzpicture}[scale=1]
\foreach \i in {0,1,2,3}
	\node[vert] (\i) at ({cos(360 * \i/4+45)}, {sin(360*\i/4+45}) {};
	\draw [dashed] (0) -- node[above] {$e'$} (1) ; 
    \draw [blue] (1) -- (2) ;
    \draw [dashed] (2) -- node[below] {$e$} (3) ; 
    \draw [blue] (3) -- (0) ;
\end{tikzpicture}
\caption{}
\end{subfigure}
\begin{subfigure}{.24\textwidth}
\centering
\begin{tikzpicture}[scale=1]
\foreach \i in {0,1,2,3}
	\node[vert] (\i) at ({cos(360 * \i/4+45)}, {sin(360*\i/4+45}) {};
	\draw [dashed] (0) -- node[above] {$e'$} (1) ; 
    \draw [dashed] (1) -- node[left]{$f$} (2) ;
    \draw [dashed] (2) -- node[below] {$e$} (3) ; 
    \draw [dashed] (3) -- node[right] {$f'$} (0) ;
\end{tikzpicture}
\caption{}
\end{subfigure}
\begin{subfigure}{.24\textwidth}
\centering
\begin{tikzpicture}[scale=1]
\foreach \i in {0,1,2,3,4}
	\node[vert] (\i) at ({cos(360 * \i/5+18)}, {sin(360*\i/5+18}) {};
	\draw [dashed] (3) -- node[below] {$e$} (4) ;
    \draw [dashed] (0) -- node[above] {$e'$} (1) ;
    \draw [dashed] (0) -- node[right] {$f$} (4) ;
    \draw [dashed] (1) -- node[above] {$f'$} (2) ;
    \draw [red] (3) -- node[right] {} (0) ;
    \draw [blue] (0) -- node[right]{}  (2) ;
    \draw [blue] (3) -- node[right]{}  (2) ;
\end{tikzpicture}
\caption{}
\end{subfigure}
\caption{Representations of a 2-conflict and two 4-conflicts in $\mc{D}$. All colored edges are from Phase A, and dashed edges will be colored in Phase B. }
\label{fig:Dconflicts}
\end{figure}

We let $\mc{D}$ be the collection of all the conflicts described above, and we let $\mc{C}_B$ be the collection of all minimal (inclusion-wise) conflicts in $\mc{D}$. Since Conditions \ref{cond:c1b}--\ref{cond:c4b} all require only upper bounds on the cardinalities of certain sets of edges, it suffices to check these conditions for $\mc{D}$ instead of $\mc{C}_B$ whenever convenient.
\begin{prop}
     A $\mc{C}_B$-free $X$-perfect matching in $\mc{H}_B$ corresponds to a $(5, 8)$-coloring of $K_n$. 
\end{prop}
\begin{proof}
    Recall that Phase A produces a partial $(5, 8)$-coloring, so in particular the Phase A coloring has at most 2 color repetitions in any set of 5 vertices. Now suppose we have a $\mc{C}_B$-free $X$-perfect matching in $\mc{H}_B$. Since our matching is $\mc{C}_B$-free it is also $\mc{D}$-free. Since $X = E'' \cup E'''$, our matching colors every edge that did not get colored in Phase A. To show that the final coloring is a $(5, 8)$-coloring, let $S$ be a set of 5 vertices. From Phase A, $S$ has either 0, 1, or 2 repeats. If $S \notin \mc{S}_{5, 0}(e) \cup \mc{S}_{5, 1}(e) \cup \mc{S}_{5, 2}(e)$ for all $e$, this can only be because $S$ does not contain two disjoint edges that are both in $E''$ or both in $E'''$ (see Definition \ref{def:Sab}). In that case it is not possible for $S$ to gain any new color repetitions in Phase B. If $S \in \mc{S}_{5, 0}(e)$ for some $e$ then $S$ cannot get 3 new repetitions in Phase B since that situation would arise from a 6-conflict (note that in a $\mc{D}$-free matching we cannot have 3 edges of the same Phase B color in $S$ since our 2-conflicts enforce that adjacent edges have different colors). If $S \in \mc{S}_{5, 1}(e)$ for some $e$ then $S$ cannot get 2 new repetitions in Phase B since that would be a 4-conflict. Finally, if $S \in \mc{S}_{5, 2}(e)$ for some $e$ then $S$ cannot get any new repetitions in Phase $B$ since that would be a 2-conflict. 
\end{proof}

We check condition \ref{cond:c1b}, starting with $j=2$. Fix $e_c$ and we will bound the number of $f_c$ such that $\{e_c, f_c\}$ is a conflict. There are $O(n^{1-\eps_A^3})$ edges $f$ adjacent to $e$ by part (\ref{lem:quasi-random-deg}), and $O(n)$ edges $f$ in some $S \in \mc{S}_{4, 1}(e) \cup \mc{S}_{5, 2}(e)$ by part (\ref{lem:quasi-random-a-b}). Thus we have $\Delta(\mc{D}^{(2)}) = O(n) =o( d_B \log d_B)$. Now we check $j=4$. We have $\Delta(\mc{D}^{(4)}) = O(n^3)=o( d_B^3 \log d_B)$ since, fixing some $e_c$, we can determine the rest of a 4-conflict by first choosing some $S \in \mc{S}_{4, 0}(e) \cup \mc{S}_{5, 1}(e)$ which are $O(n^2)$ by part (\ref{lem:quasi-random-a-b}) and then choosing a color $c'$.  For $j=6$ using part (\ref{lem:quasi-random-a-b}) we similarly have  $\Delta(\mc{D}^{(6)}) = O(n^5) =o( d_B^5 \log d_B)$. Thus \ref{cond:c1b} holds for any positive $\alpha$.

Now we check \ref{cond:c2b}. For $j=2$ there is no possible value for $j'$ so the statement is vacuously true. For $j=4$ it suffices to check that we have
\[
\Delta_{2}(\mc{C}_B^{(4)}) =O(n), \qquad \Delta_{3}(\mc{C}_B^{(4)}) =O(n^{1-\eps_A^3}).
\]
Indeed, consider the case of a 4-conflict in $\mc{C}_B$ arising from some $S \in \mc{S}_{4, 0}(e)$. If we fix $j'=2$ of these colored edges then we have fixed either $2$ colors and $3$ vertices or $1$ color and $4$ vertices, and in either case there are $O(n)$ ways to complete the alternating 4-cycle. If we fix $j'=3$ colored edges they are in at most one alternating 4-cycle. Now consider the case of a 4-conflict in $\mc{C}_B$ arising from some $S \in \mc{S}_{5, 1}(e)$. We can assume that this conflict does not make an alternating 4-cycle since we just handled that case (then we would have $S \in \mc{S}_{4, 0}(e)$). Thus the 4 colored edges in $K_n$ corresponding this conflict form an alternating 4-path. If we fix $j'=2$ of these colored edges there are $O(n)$ ways to complete the conflict. If we fix $j'=3$ of them then by Lemma \ref{lem:quasi-random} part (\ref{lem:quasi-random-deg}) there are $O(n^{1-\eps_A^3})$ ways to complete the conflict. Thus we have checked \ref{cond:c2b} for $j=4$ (this condition just has to hold for some $\eps=\eps_B>0$ and we can choose $\eps_B \ll \eps_A, \d$). 

Now we check \ref{cond:c2b} for $j=6$. It suffices to check that 
\[
\Delta_{2}(\mc{C}_B^{(6)}) =O(n^{3}), \qquad \Delta_{3}(\mc{C}_B^{(6)}) =O(n^{2}), \qquad \Delta_{4}(\mc{C}_B^{(6)}) =O(n), \qquad \Delta_{5}(\mc{C}_B^{(6)}) =O(1).
\]
Indeed, any 6-conflict corresponds to 3 pairs of edges, where each pair is a monochromatic matching. Any two of these pairs forms an alternating 4-path (if they formed an alternating 4-cycle then this conflict would not be minimal). Together, these edges have 3 distinct colors and span 5 vertices. If we fix $j'=2$ colored edges in such a conflict then we have fixed either 3 vertices and 2 colors or 4 vertices and 1 color, justifying the first bound above. Fixing $j'=3$ colored edges fixes at least 4 vertices and 2 colors, justifying the second bound above. If we fix $j'=4$ colored edges in our conflict, then we have fixed all 5 vertices and 2 colors or 4 vertices and 3 colors, justifying the third bound. Finally for $j'=5$ we fix all 5 vertices and 3 colors. Thus \ref{cond:c2b} holds. 

Now we check \ref{cond:c3b}. We fix a vertex and a non-incident edge in $\mc{H}_B$. If our vertex is in $Y$ then it is in at most one edge of $\mc{H}_B$ and we are done. So assume our fixed vertex is say $e \in X$, and our fixed edge is say $f_c$. The only edge containing $e$ in $\mc{H}_B$ that could form a conflict with $f_c$ is $e_c$ since 2-conflicts only have one color. Thus \ref{cond:c3b} holds. 

Now we check \ref{cond:c4b}. Fix $e_c, e'_c \in E(\mc{H}_B)$. By Lemma \ref{lem:quasi-random} part (\ref{lem:quasi-random-x-lem}), 
\[
\big|\big\{e''_c \in E(\mc{H}_B) : \{e_c, e''_c\}\in \mc{D} \mbox{ and } \{e'_c, e''_c\}\in \mc{D}\big\}\big| \le n^{1-\delta/2}.
\]
Thus \ref{cond:c4b} holds. 

By Theorem \ref{thm:B}, we get a $\mc{C}_B$-free $X$-perfect matching of $\mc{H}_B$, which corresponds to a way to color the remaining uncolored edges using $|C_{B1}|+|C_{B2}| = 4n \log^{-1/10} = o(n)$ colors. In Phase A we used $6n/7 +o(n)$ colors. Thus we have proved Theorem \ref{thm:main}. 

\section{Concluding remarks}
As we were finalizing this manuscript for submission, we noticed that Joos, Mubayi and Smith \cite{JMS24} very recently proved a theorem which can be applied as a black box to do two-phase colorings all at once. This black box theorem can be used to reprove several of the recent generalized Ramsey results that appeared in \cite{BBHZ24, BCDP22, BHZ23, GHPSZ23, LM24}. 

However, it seems to us that the black box in \cite{JMS24} is not quite sufficient to obtain the main result of this paper (although a slightly improved version might be sufficient). Indeed, the proof method in \cite{JMS24} is essentially a generalized two-phase proof where the first phase is an application of (a version of) Theorem \ref{thm:A} and the second phase is the Local Lemma. As we pointed out, such a proof technique would not work for us which is why our second phase was instead an application of Theorem \ref{thm:B}. The difficulty seems to be related to the sharable edges, which are a feature that was not present in any of the recent generalized Ramsey results.  Thus perhaps it is not surprising that when we try to input our setup into the new result in \cite{JMS24}, it looks like it does not satisfy all the conditions for the theorem to apply. 

We believe the framework from \cite{JMS24} can be modified in the spirit of our proof technique to obtain an improved result that would reprove our main result.  For example, we find it likely one could apply our present proof technique (i.e.\ a two-phase coloring applying Theorem \ref{thm:A} and then Theorem \ref{thm:B} instead of using the Local Lemma to finish) in a more general setting to obtain a result that could be applied to reprove Theorem \ref{thm:main}.


\appendix
\section{Quasirandom properties from Phase A}\label{sec:quasirandom}

 Our main tools will be some extensions to Theorem \ref{thm:A}. We now start introducing these tools. 

Let $w\colon \binom{\mc{H}}{j}\rightarrow[0,\ell]$ where $j\in  \mathbb{N}$. We call $w$ a \emph{test function} for $\mc{H}$ if $w(F)=0$ whenever $F\in\binom{\mc{H}}{j}$ is not a matching.
We say that $w$ is $j$-uniform.
In general, for a function $w\colon A\rightarrow\mathbb{R}$ and a finite set $X\subset A$, we define $w(X):=\sum_{x\in X}w(x)$.
For a $j$-uniform test function $w$, we also use $w$ to denote the extension of $w$ to arbitrary subsets $F \subseteq E(\mc{H})$ where we define $w(F):=w(\binom{F}{j}) = \sum_{F' \in \binom{F}{j}}w(F')$.
We will see that we can formally state many quasirandom properties of our matching $\mc{M}_A$ by defining an appropriate test function $w$ and then estimating $w(\mc{M}_A)$. To estimate it, we have the following extension of Theorem \ref{thm:B}. We use Theorem~2.1 in \cite{JM22} which is the straightforward adaptation of Theorem~3.3 in \cite{GJKKL} to the setting of Theorem \ref{thm:A}.

\begin{theorem}[{\cite[Theorem~2.1]{JM22}}]\label{thm:test}
	For all $k,\ell\geq 2$, there exists $\eps_0>0$ such that for all $\eps\in(0,\eps_0)$, there exists $d_0$ such that the following holds for all $d\geq d_0$.
	Suppose $\mc{C}$ is an $\ell$-bounded conflict system on a $k$-uniform hypergraph $\mc{H}$ satisfying all the assumptions of Theorem \ref{thm:A}.
	Suppose $\mathscr{W}$ is a set of test functions of uniformity at most $\ell$ with $|\mathscr{W}|\leq \exp(d^{\eps^3})$ where for each $j$-uniform $w \in \mathscr{W}$ we have
 \begin{enumerate}[label=\textup{(W\arabic*)}]
		\item\label{item: trackable size} $w(\mc{H})\geq d^{j+\eps}$;
		\item\label{item: trackable degrees} $w(\{ F\in\binom{\mc{H}}{j} \colon F\supseteq F' \} )\leq w(\mc{H})/d^{j'+\eps}$ for all $j'\in[j-1]$ and $F'\in\binom{\mc{H}}{j'}$;
		\item\label{item: trackable neighborhood} $|(\mc{C}_e)^{(j')}\cap (\mc{C}_f)^{(j')}|\leq d^{j'-\eps}$ for all $e,f\in \mc{H}$ with $w(\{F\in\binom{\mc{H}}{j}\colon e,f\in F\})>0$ 
		and all $j'\in[\ell-1]$;
		\item\label{item: trackable no conflicts} $w(F)=0$ for all $F\in\binom{\mc{H}}{j}$ that are not $\mc{C}$-free.
\end{enumerate}
	Then, there exists a $\mc{C}$-free matching $\mc{M}\subset \mc{H}$ of size at least $(1-d^{-\eps^3})n/k$ with $w(\mc{M})=(1\pm d^{-\eps^3})d^{-j}w(\mc{H})$ for all $j$-uniform $w\in\mathscr{W}$.
\end{theorem}

 Recall that $e \in E'' \cup E'''$, let $\mc{S}_{a, b}(e)$ be the family of sets $S \subseteq V(K_n)$ with $|S|=a$, having $b$ color repetitions in the coloring from Phase A, and containing an uncolored edge $e'$ which is disjoint from $e$ and which belongs to the same set (either $E''$ or $E'''$) as $e$.

\setcounter{lemma}{0}
\begin{lemma}
Our matching $\mc M_A$ and corresponding partial $(5,8)$-coloring from Phase A can be taken to satisfy the following:
\begin{enumerate}
\item  each vertex $v$ has uncolored degree $O(n^{1-\eps_A^3})$; 
\item for each shareable edge $uv \in E''$, the number of gadgets containing $uv$ as in Figure~\ref{fig:scaff} is $O(n^\delta)$ 
\item for each uncolored edge $uv$ and every $(a, b) \in \{(4, 0), (4, 1), (5,0), (5, 1), (5, 2)\}$, we have $|\mc{S}_{a, b}(uv)|=O(n^{a-b-2})$
\item  For each pair $e_1, e_2$ of uncolored edges, there are at most $O(n^{1-\delta/2})$ pairs of sets $S_1, S_2$ with the following properties. $S_j \in \mc{S}_{4,1}(e_j) \cup \mc{S}_{5,2}(e_j)$ for $j=1,2$. Furthermore there is a single edge $e' \subseteq S_1 \cap S_2$ which is playing the role of the second uncolored edge (besides $e_1$) in $S_1$ and (besides $e_2$) in $S_2$. 
\end{enumerate}
\end{lemma}

\subsection{Proof of Lemma~\ref{lem:quasi-random}(\ref{lem:quasi-random-deg})}

For fixed $v$, let $w: E(\mc H_A) \rightarrow [0,3]$ be the 1-uniform test function that assigns an edge $\Gamma \in E(\mc H_A)$ to the number of colored edges of $K_n$ in $\Gamma$ that contain $v$.

Next we estimate $w(\mc{H}_A)$. In the following calculation, the first term arises when $v = \gamma_i$ for $i=1,2,5,6$, which is adjacent to two colored edges (one from $C_{A1}$ and one from $C_{A2}$). The second term arises when $v=\gamma_i$ for $i=3,4$, which is adjacent to three colored edges (refer to Figure~\ref{fig:G}). So using McDiarmid's inequality gives us
\begin{align*}
w(\mc{H}_A) &= 2\cdot n^3 \binom {n}2 |C_{A1}||C_{A2}|^2 p^{14} + 3 \cdot (n-1) \binom{n-2}{2}^2 |C_{A1}||C_{A2}|^2 p^{14}\pm n^8p^{14.5} \\
&= \frac{27}{196}n^8p^{14} \pm n^8p^{14.5}.
\end{align*}
Then 
Clearly, $w(\mc{H}_A)=\Omega(d_A^n)\ge d_A^{1+\varepsilon_A}$, verifying \ref{item: trackable size}. Notice that since $j=1$, conditions \ref{item: trackable degrees}--\ref{item: trackable no conflicts} are satisfied trivially. 

Thus, we have 
$$w(\mc{M}_A)=(1\pm d_A^{-\varepsilon_A^3})\frac{\frac{27}{196}n^8p^{14} \pm n^8p^{14.5}}{\frac{27}{196}n^7p^{14} \pm n^7p^{14.5}}$$

For sufficiently large $n$, by \eqref{eqn:ed}, we have that $w(\mc{M}_A)= (1\pm n^{-\eps_A^3})n$.
Thus, for any fixed vertex, the uncolored degree will be $(n-1)-(1\pm n^{-\eps_A^3})n=O(n^{1-\eps_A^3})$.

\subsection{Proof of Lemma~\ref{lem:quasi-random}(\ref{lem:quasi-random-shareable})}

For fixed $u,v$, let $w: E(\mc H_A) \rightarrow [0,1]$ to be the test function that assigns an edge $\Gamma \in E(\mc H_A)$ to $1$ if $uv$ is a shareable edge of $\Gamma$ of the form $\gamma_i\gamma_j$ for $ij=14, 24, 35, 36$. Then 
$$
w(\mc{H}_A) = \Theta(n^4n^3p^{13}) = \Theta(d_A/p)
$$
and by \eqref{eqn:ed} we have $w(\mc{H}_A) = \Omega(d_An^{\delta}) \ge d_A^{1+\varepsilon_A}$, verifying \ref{item: trackable size}. Since $j=1$, conditions \ref{item: trackable degrees}--\ref{item: trackable no conflicts} are satisfied trivially.

Therefore, 
$$
w(\mc{M}_A)=(1\pm d_A^{-\varepsilon_A^3})\frac{O(d_A/p)}{d_A} = O(n^{\delta}).
$$

\subsection{Proof of Lemma~\ref{lem:quasi-random}(\ref{lem:quasi-random-a-b})}

\begin{figure}[ht]
\begin{subfigure}{.24\textwidth}
\centering
\begin{tikzpicture}[scale=1]
\foreach \i in {0,1,2,3}
	\node[vert] (\i) at ({cos(360 * \i/4+45)}, {sin(360*\i/4+45}) {};
\foreach \x/\y in {0/1}
	\draw [dashed] (\x) -- node[right] {} (\y) ; 
\node[vert, label=left:$u$] (u) at (2) {};
\node[vert, label=left:$v$] (v) at (3) {};
\end{tikzpicture}
\caption{}
\end{subfigure}
\vspace{3ex}
\begin{subfigure}{.24\textwidth}
\centering
\begin{tikzpicture}[scale=1]
\foreach \i in {0,1,2,3,4}
	\node[vert] (\i) at ({cos(360 * \i/5+18)}, {sin(360*\i/5+18}) {};
\foreach \x/\y in {0/1, 1/2, 0/2}
	\draw [dashed] (\x) -- node[right] {} (\y) ; 
\node[vert, label=left:$u$] (u) at (3) {};
\node[vert, label=right:$v$] (v) at (4) {};
\end{tikzpicture}
\caption{}
\end{subfigure}
\begin{subfigure}{.24\textwidth}
\centering
\begin{tikzpicture}[scale=1]
\foreach \i in {0,1,2,3}
	\node[vert] (\i) at ({cos(360 * \i/4+45)}, {sin(360*\i/4+45}) {};
\foreach \x/\y in {0/3, 1/2}
	\draw [blue] (\x) -- node[right] {} (\y) ; 
\foreach \x/\y in {0/1}
	\draw [dashed] (\x) -- node[right] {} (\y) ; 
\node[vert, label=left:$u$] (u) at (2) {};
\node[vert, label=left:$v$] (v) at (3) {};
\end{tikzpicture}
\caption{}
\end{subfigure}
\begin{subfigure}{.24\textwidth}
\centering
\begin{tikzpicture}[scale=1]
\foreach \i in {0,1,2,3,4}
	\node[vert] (\i) at ({cos(360 * \i/5+18)}, {sin(360*\i/5+18}) {};
\foreach \x/\y in {1/2, 2/3}
	\draw [blue] (\x) -- node[right] {} (\y) ; 
\foreach \x/\y in {0/1, 0/2}
	\draw [dashed] (\x) -- node[right] {} (\y) ; 
\foreach \x/\y in {1/3}
	\draw [black] (\x) -- node[right] {} (\y) ;
\node[vert, label=left:$u$] (u) at (3) {};
\node[vert, label=right:$v$] (v) at (4) {};
\end{tikzpicture}
\caption{}
\end{subfigure}
\begin{subfigure}{.24\textwidth}
\centering
\begin{tikzpicture}[scale=1]
\foreach \i in {0,1,2,3,4}
	\node[vert] (\i) at ({cos(360 * \i/5+18)}, {sin(360*\i/5+18}) {};
\foreach \x/\y in {2/3, 0/4}
	\draw [blue] (\x) -- node[right] {} (\y) ; 
\foreach \x/\y in {0/1, 0/2, 1/2}
	\draw [dashed] (\x) -- node[right] {} (\y) ; 
\node[vert, label=left:$u$] (u) at (3) {};
\node[vert, label=right:$v$] (v) at (4) {};
\end{tikzpicture}
\caption{}
\end{subfigure}
\vspace{3ex}
\begin{subfigure}{.24\textwidth}
\centering
\begin{tikzpicture}[scale=1]
\foreach \i in {0,1,2,3,4}
	\node[vert] (\i) at ({cos(360 * \i/5+18)}, {sin(360*\i/5+18}) {};
\foreach \x/\y in {1/2, 0/4}
	\draw [blue] (\x) -- node[right] {} (\y) ; 
\foreach \x/\y in {0/1, 0/2}
	\draw [dashed] (\x) -- node[right] {} (\y) ; 
\node[vert, label=left:$u$] (u) at (3) {};
\node[vert, label=right:$v$] (v) at (4) {};
\end{tikzpicture}
\caption{}
\end{subfigure}
\begin{subfigure}{.24\textwidth}
\centering
\begin{tikzpicture}[scale=1]
\foreach \i in {0,1,2,3,4}
	\node[vert] (\i) at ({cos(360 * \i/5+18)}, {sin(360*\i/5+18}) {};
\foreach \x/\y in {1/2, 2/3}
	\draw [blue] (\x) -- node[right] {} (\y) ; 
\foreach \x/\y in {1/3}
	\draw [black] (\x) -- node[right] {} (\y) ;
\draw [blue] (4) -- (0) ;
\foreach \x/\y in {0/1, 0/2}
	\draw [dashed] (\x) -- node[right] {} (\y) ; 
\node[vert, label=left:$u$] (u) at (3) {};
\node[vert, label=right:$v$] (v) at (4) {};
\end{tikzpicture}
\caption{}
\end{subfigure}
\begin{subfigure}{.24\textwidth}
\centering
\begin{tikzpicture}[scale=1]
\foreach \i in {0,1,2,3,4}
	\node[vert] (\i) at ({cos(360 * \i/5+18)}, {sin(360*\i/5+18}) {};
\foreach \x/\y in {1/2, 2/3}
	\draw [blue] (\x) -- node[right] {} (\y) ; 
\foreach \x/\y in {0/1, 4/0}
	\draw [red] (\x) -- node[right] {} (\y) ; 
\foreach \x/\y in {0/2}
	\draw [dashed] (\x) -- node[right] {} (\y) ;
 \foreach \x/\y in {1/3, 1/4}
	\draw [black] (\x) -- node[right] {} (\y) ;
\node[vert, label=left:$u$] (u) at (3) {};
\node[vert, label=right:$v$] (v) at (4) {};
\end{tikzpicture}
\caption{}
\end{subfigure}
\begin{subfigure}{.24\textwidth}
\centering
\begin{tikzpicture}[scale=1]
\foreach \i in {0,1,2,3,4}
	\node[vert] (\i) at ({cos(360 * \i/5+18)}, {sin(360*\i/5+18}) {};
\foreach \x/\y in {1/2, 2/3}
	\draw [blue] (\x) -- node[right] {} (\y) ; 
\foreach \x/\y in {1/3, 4/0}
	\draw [red] (\x) -- node[right] {} (\y) ; 
\foreach \x/\y in {0/1, 0/2}
	\draw [dashed] (\x) -- node[right] {} (\y) ; 
\node[vert, label=left:$u$] (u) at (3) {};
\node[vert, label=right:$v$] (v) at (4) {};
\end{tikzpicture}
\caption{}
\end{subfigure}
\begin{subfigure}{.24\textwidth}
\centering
\begin{tikzpicture}[scale=1]
\foreach \i in {0,1,2,3,4}
	\node[vert] (\i) at ({cos(360 * \i/5+18)}, {sin(360*\i/5+18}) {};
\foreach \x/\y in {1/2, 2/3}
	\draw [blue] (\x) -- node[right] {} (\y) ; 
\foreach \x/\y in {1/4, 3/0}
	\draw [red] (\x) -- node[right] {} (\y) ;
\foreach \x/\y in {0/1, 0/2}
	\draw [dashed] (\x) -- node[right] {} (\y) ; 
\foreach \x/\y in {1/3}
	\draw [black] (\x) -- node[right] {} (\y) ;
\node[vert, label=left:$u$] (u) at (3) {};
\node[vert, label=right:$v$] (v) at (4) {};
\end{tikzpicture}
\caption{}
\end{subfigure}
\begin{subfigure}{.24\textwidth}
\centering
\begin{tikzpicture}[scale=1]
\foreach \i in {0,1,2,3,4}
	\node[vert] (\i) at ({cos(360 * \i/5+18)}, {sin(360*\i/5+18}) {};
\foreach \x/\y in {1/2, 2/3}
	\draw [blue] (\x) -- node[right] {} (\y) ; 
\foreach \x/\y in {0/1, 2/4}
	\draw [red] (\x) -- node[right] {} (\y) ; 
 \foreach \x/\y in {0/2}
	\draw [dashed] (\x) -- node[right] {} (\y) ;
\foreach \x/\y in {1/3}
	\draw [black] (\x) -- node[right] {} (\y) ;
\node[vert, label=left:$u$] (u) at (3) {};
\node[vert, label=right:$v$] (v) at (4) {};
\end{tikzpicture}
\caption{}
\end{subfigure}
\vspace{3ex}
\begin{subfigure}{.24\textwidth}
\centering
\begin{tikzpicture}[scale=1]
\foreach \i in {0,1,2,3,4}
	\node[vert] (\i) at ({cos(360 * \i/5+18)}, {sin(360*\i/5+18}) {};
\foreach \x/\y in {1/2, 2/3}
	\draw [blue] (\x) -- node[right] {} (\y) ; 
\foreach \x/\y in {0/2, 1/4}
	\draw [red] (\x) -- node[right] {} (\y) ; 
\foreach \x/\y in {0/1}
	\draw [dashed] (\x) -- node[right] {} (\y) ;
\foreach \x/\y in {1/3}
	\draw [black] (\x) -- node[right] {} (\y) ;
\node[vert, label=left:$u$] (u) at (3) {};
\node[vert, label=right:$v$] (v) at (4) {};
\end{tikzpicture}
\caption{}
\end{subfigure}
\begin{subfigure}{.24\textwidth}
\centering
\begin{tikzpicture}[scale=1]
\foreach \i in {0,1,2,3,4}
	\node[vert] (\i) at ({cos(360 * \i/5+18)}, {sin(360*\i/5+18}) {};
\foreach \x/\y in {0/2, 0/3}
	\draw [blue] (\x) -- node[right] {} (\y) ; 
\foreach \x/\y in {2/3, 0/1}
	\draw [red] (\x) -- node[right] {} (\y) ; 
\foreach \x/\y in {1/2}
	\draw [dashed] (\x) -- node[right] {} (\y) ; 
\node[vert, label=left:$u$] (u) at (3) {};
\node[vert, label=right:$v$] (v) at (4) {}; 
\end{tikzpicture}
\caption{}
\end{subfigure}
\begin{subfigure}{.24\textwidth}
\centering
\begin{tikzpicture}[scale=1]
\foreach \i in {0,1,2,3,4}
	\node[vert] (\i) at ({cos(360 * \i/5+18)}, {sin(360*\i/5+18}) {};
\foreach \x/\y in {0/2, 0/3}
	\draw [blue] (\x) -- node[right] {} (\y) ; 
\foreach \x/\y in {2/3, 4/0}
	\draw [red] (\x) -- node[right] {} (\y) ; 
\foreach \x/\y in {0/1, 1/2}
	\draw [dashed] (\x) -- node[right] {} (\y) ; 
\node[vert, label=left:$u$] (u) at (3) {};
\node[vert, label=right:$v$] (v) at (4) {}; 
\end{tikzpicture}
\caption{}
\end{subfigure}
\begin{subfigure}{.24\textwidth}
\centering
\begin{tikzpicture}[scale=1]
\foreach \i in {0,1,2,3,4}
	\node[vert] (\i) at ({cos(360 * \i/5+18)}, {sin(360*\i/5+18}) {};
\foreach \x/\y in {2/4, 1/0}
	\draw [blue] (\x) -- node[right] {} (\y) ; 
\foreach \x/\y in {2/3, 4/0}
	\draw [red] (\x) -- node[right] {} (\y) ; 
\foreach \x/\y in {0/2, 1/2}
	\draw [dashed] (\x) -- node[right] {} (\y) ; 
\node[vert, label=left:$u$] (u) at (3) {};
\node[vert, label=right:$v$] (v) at (4) {}; 
\end{tikzpicture}
\caption{}
\end{subfigure}
\begin{subfigure}{.24\textwidth}
\centering
\begin{tikzpicture}[scale=1]
\foreach \i in {0,1,2,3,4}
	\node[vert] (\i) at ({cos(360 * \i/5+18)}, {sin(360*\i/5+18}) {};
\foreach \x/\y in {2/3, 1/0}
	\draw [blue] (\x) -- node[right] {} (\y) ; 
\foreach \x/\y in {1/2, 4/0}
	\draw [red] (\x) -- node[right] {} (\y) ; 
\foreach \x/\y in {0/2}
	\draw [dashed] (\x) -- node[right] {} (\y) ; 
\node[vert, label=left:$u$] (u) at (3) {};
\node[vert, label=right:$v$] (v) at (4) {}; 
\end{tikzpicture}
\caption{}
\end{subfigure}
\caption{Cases for $S \in \mc{S}_{a, b}(uv)$. Dashed edges could play the role of the second uncolored edge $e'$. Triangles are shown with one possible orientation but other orientations are possible. Black is used for edges in triangles whose color does not repeat within $S$.}
\label{fig:ab-types}
\end{figure}
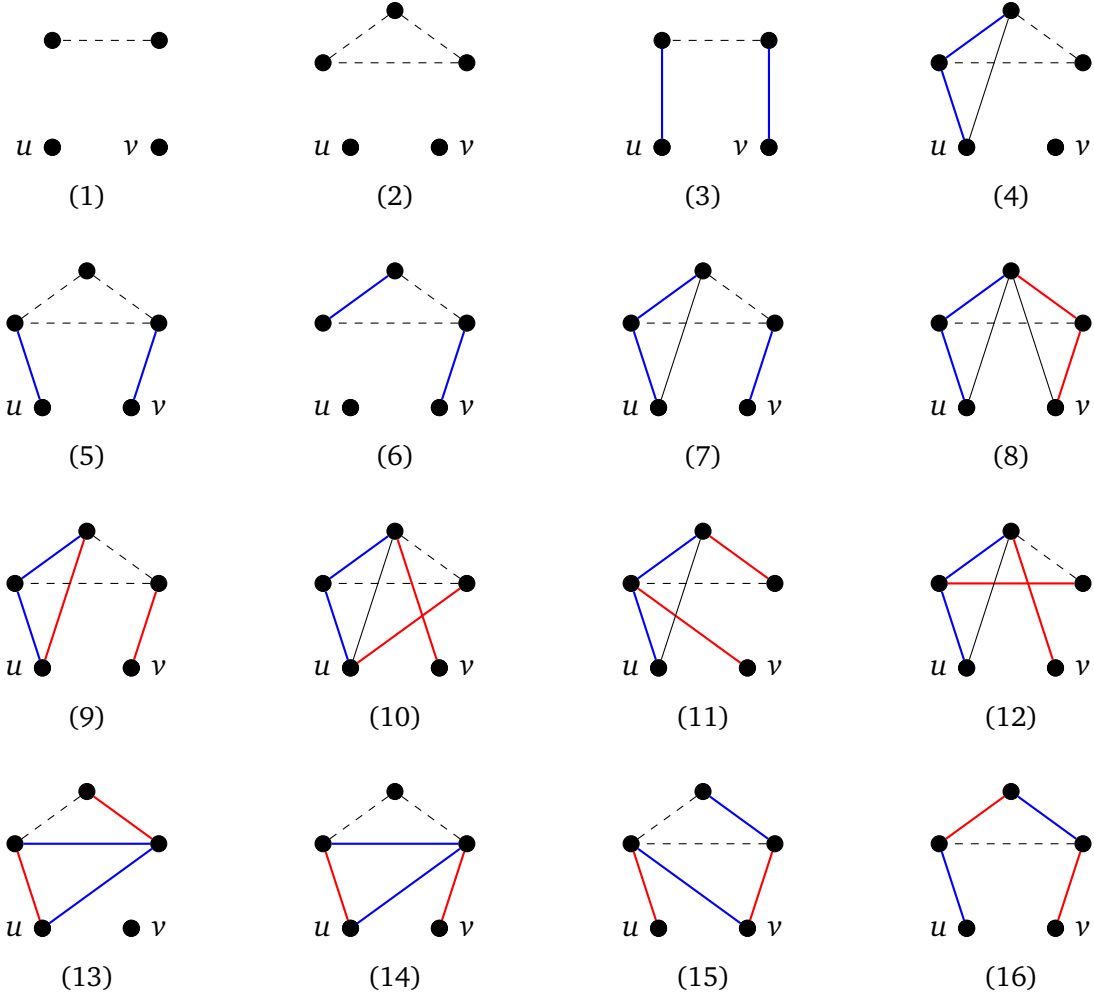

Figure \ref{fig:ab-types} shows representative sets $S \in \mc{S}_{a, b}(uv)$ for all relevant pairs $(a, b)$. Some pairs $(a, b)$ have several subcases, for example the possible cases for pair $(5, 2)$ are in Subfigures (7)--(16). 

Almost all of the cases follow from some easy facts: we have $n$ vertices, $O(n)$ colors, and every vertex has degree at most 2 in every color. Sometimes we also use the fact that between any two vertices there is at most one monochromatic 2-path. Indeed, to see $|\mc{S}_{4, 0}(uv)| = O(n^2)$ and $|\mc{S}_{5, 0}(uv)|=O(n^3)$ (Figures \ref{fig:ab-types}(1) and (2)) we just choose the rest of the vertices in $S$.  To see $|\mc{S}_{4, 1}(uv)| = O(n)$ (Figures \ref{fig:ab-types}(3)) we choose a color which determines $O(1)$ possible sets $S$. Similarly we can see $|\mc{S}_{5, 1}(uv)| = O(n^2)$ (Figures \ref{fig:ab-types}(4)--(6)). The last bound for us to show is $|\mc{S}_{5, 2}(uv)|=O(n)$ (Figures \ref{fig:ab-types}(7)--(16)). The cases shown in (7)--(13) all follow from the easy facts we have used so far. The case in (14)  can only apply when $uv \in E''$ since $uv$ is used as a sharable edge by a gadget. To bound the number of such sets $S$, first we choose the gadget in $O(n^\d)$ ways by Lemma \ref{lem:quasi-random}(\ref{lem:quasi-random-shareable}), and then we choose a sharable edge incident with one of the vertices in the gadget. Since each vertex has sharable degree $O(n^{1-\d})$ the number of such sets $S$ is $O(n^{\d} \cdot n^{1-\d})= O(n)$ so the case in (14) is done. Thus we arrive at the case in (16). Bounding such sets $S$ will require a fresh application of Theorem \ref{thm:test}.

\begin{figure}[ht]
\centering
\begin{tikzpicture}[scale=1]
\begin{scope}[shift={(-1.955,1.425)}, rotate=-90-72]
	\DrawGamma{blue!25}{black!25}{black!25}{$\Gamma_1$}
\end{scope}
\node[label=below:$u$] (u) at (g1) {};
\draw[blue, line width=2pt] (g1) -- (g2);

\begin{scope}[shift={(-1.21,3.73)}, rotate=-90-72-72]
	\DrawGamma{red!25}{black!25}{black!25}{$\Gamma_2$}
\end{scope}
\draw[red, line width=2pt] (g1) -- (g2);

\begin{scope}[shift={(1.21,3.73)}, rotate=-90-72-72-72]
	\DrawGamma{blue!25}{black!25}{black!25}{$\Gamma_3$}
\end{scope}
\draw[blue, line width=2pt] (g1) -- (g2);

\begin{scope}[shift={(1.955,1.425)}, rotate=-90-72-72-72-72]
	\DrawGamma{red!25}{black!25}{black!25}{$\Gamma_4$}
\end{scope}
\draw[red, line width=2pt] (g1) -- (g2);
\node[label=below:$v$] (v) at (g2) {};
\draw[dashed, line width=2pt] (u) -- (v);
\end{tikzpicture}
\caption{Structures counted by $w$.}
\label{fig:orientations-5}
\end{figure}
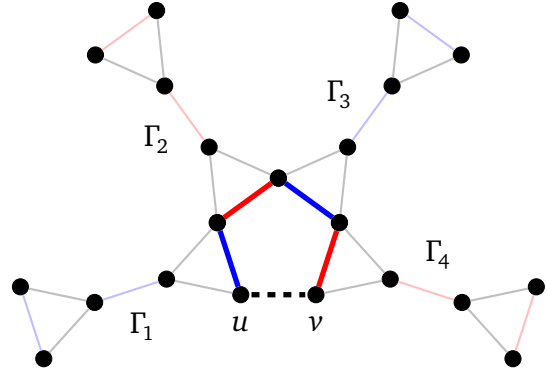

To bound the number of sets $S$ as in Figure \ref{fig:ab-types}(16), our test function $w$ is counting 4-tuples of gadgets like Figure \ref{fig:orientations-5} (and so $j=4$). $u, v$ are fixed, and so we have to choose 19 vertices and 10 colors, and each gadget contributes a $p^{14}$ factor for all the sharable edges and nonhitting colors. Thus
$$
w(\mc{H}_A) = \Theta(n^{19}n^{10}p^{4\cdot 14})=\Theta(d_A^4n),
$$
 verifying \ref{item: trackable size}. Next we verify \ref{item: trackable degrees}. The reader can check that if we fix any set $F'$ of $j' \in \{1, 2, 3\}$ of our gadgets, the total number of fixed vertices plus colors is always at least $7j'+3$. Since $u, v$ were already fixed that means we have at least $7j'+1$ additional fixed vertices or colors. To complete the 4-tuple of gadgets we then have to choose a total of $28-7j'$ vertices and colors. Thus 
 $$w\left(\left\{ F\in\binom{\mc{H}_A}{j} \colon F\supseteq F' \right\}\right) = O(n^{28-7j'}) \le \frac{w(\mc{H}_A)}{d^{j'+\eps}},
 $$
 where the last inequality follows from \eqref{eqn:ed} and the fact that the power of $n$ on the right hand side is $n^{29-7j'}$. 

 To verify \ref{item: trackable neighborhood}, we let $\G$ and $\G'$ be gadgets such that  with $w\left(\{ F \in \binom{\mc{H}_A}{j} : \G, \G' \in F\}\right) > 0$.  Let $\Gamma_1, \ldots, \Gamma_{j'}$ be a set of $j'$ edges in $\mc{H}_A$ such that both $C_\G :=\{\G, \Gamma_1, \ldots, \Gamma_{j'}\}$ and $C_{\G'}:=\{\G', \Gamma_1, \ldots, \Gamma_{j'}\}$ are conflicts. Let $S_\G$ and $S_{\G'}$ be the corresponding colored subgraph in $K_n$ isomorphic to one of the graphs in Figure~\ref{fig:conflict}. Observe that $\G$ (res.\ $\G'$) has to contribute at least one colored edge to $S_\G$ (res.\ $S_{\G'}$). Thus if we look at the set of all vertices in $K_n$ spanned by the gadgets  $\Gamma_1, \ldots, \Gamma_{j'}$ it contains an edge colored by $\G$ and a different edge colored by $\G'$. We know that each conflict of size $j'+1$ involves a total of $7(j'+1)+2$ vertices and colors. Fixing $\G$ determines a total of 9 vertices and colors used in the conflict $C_\G$, and fixing $\G'$ determines at least one additional vertex, so there are at most $7j'-1$ choices of vertices and colors to complete the conflict $C_\G$. Thus  $|((\mc{C}_A)_\G)^{(j')}\cap ((\mc{C}_A)_{\G'})^{(j')}| = O(n^{7j'-1}) < d_A^{j'-\varepsilon}$, verifying \ref{item: trackable neighborhood}.

Finally, \ref{item: trackable no conflicts} is satisfied since by definition $w$ assigns positive weight only for $E \in \binom{\mc{C}_A}{4}$ that are $\mc{C}_A$-free.

Thus, by Theorem~\ref{thm:test}, we have $$w(\mc{M}_A)=(1\pm d_A^{-\varepsilon^3})\frac{\Theta(d_A^4n)}{d_A^{4}} = \Theta(n).$$

This completes the last case (16) for the proof of  Lemma~\ref{lem:quasi-random}(\ref{lem:quasi-random-a-b}).

\subsection{Proof of Lemma~\ref{lem:quasi-random}(\ref{lem:quasi-random-x-lem})}

To prove this part we will need a lemma from \cite{GJKKL}, but first a definition. We say that a $j$-uniform hypergraph $\mc{X}$ is {\em $(D, r)$-spread} if we have $\Delta_{j'}(\mc{X}) \le Dr^{j'}$ for each $j'=0, \ldots j-1$. The following is essentially Lemma 6.5 from \cite{GJKKL} where we take $s=0$. 

\begin{lemma}[{\cite[Lemma~6.5]{GJKKL}}]\label{lem:X}
		Suppose $\mathfrak{X}$ is a family of uniform hypergraphs on vertex set $E(\mc{H})$, each having uniformity at most $2\ell$. Suppose each $j$-uniform $\mc{X} \in \mathfrak{X}$ is~$(D,1/d)$-spread with $D \ge d^j$. Suppose $|\mathfrak{X}| \le \exp\left(d^{\eps/300 \ell} \right)$. Then in Theorems \ref{thm:A} and \ref{thm:test} it is possible to take $\mc{M}$ such that for all $\mc{X} \in \mathfrak{X}$, $\mc{M}$ contains at most $Dd^{-j+\eps/12}$ edges of $\mc{X}$. 
	\end{lemma}

 The statement of the lemma in \cite{GJKKL} is for a single hypergraph $\mc{X}$ and it gives a probability bound for the event that $\mc{M}$ (which is the output of a random process in that paper) contains more than $Dd^{-j+\eps/12}$ edges of $\mc{X}$. The version we stated above follows from the probability bound given in \cite{GJKKL} and applying the union bound to our family $\mathfrak{X}$.

 Fix $e_1, e_2$. We will bound the number of uncolored edges $e'$ as described in Lemma~\ref{lem:quasi-random}(\ref{lem:quasi-random-x-lem}) by considering several cases. In each case we will consider a family of hypergraphs $\mc{X}$ satisfying the conditions of Lemma \ref{lem:X}. Since each family of hypergraphs will be only polynomial in size and there are a constant number of these families, Lemma~\ref{lem:quasi-random}(\ref{lem:quasi-random-x-lem}) will follow. 

\begin{figure}[ht]
\begin{subfigure}{.4\textwidth}
\centering
\begin{tikzpicture}[scale=1]
\begin{scope}[rotate=-18, shift={(0,0)}]
	\DrawGamma{black}{black}{black}{}
\end{scope}
\node (u) at (g3) {};
\node (v) at (g6) {};
\draw[dashed] (u) --node[midway,label=below:$e_1$]{} (v);
\begin{scope}[rotate=18, shift={(2.77,-.90)}]
	\DrawGamma{black}{black}{black}{}
\end{scope}
\node(u') at (g2) {};
\node (v') at (g4) {};
\draw[dashed] (u') -- node[midway,label=below:$e_2$]{} (v');

\node (a) at (g1) {};
\node[vert] (b) at ([yshift=40]a) {};
\draw[dashed] (a) -- node[midway,label=left:$e'$]{} (b);

\draw[thin, gray] ([shift={(.15,.15)}]b) rectangle ([shift={(-.15,-1.5)}]u);
\node[label=below:$S_1$] at ([shift={(-.15,-1.5)}]u) {};
\draw[thin, gray] ([shift={(-.15,.25)}]b) rectangle ([shift={(.15,-1.6)}]v');
\node[label=below:$S_2$] at ([shift={(.15,-1.6)}]v') {};
\end{tikzpicture}
\caption{}
\end{subfigure}
\begin{subfigure}{.4\textwidth}
\centering
    \begin{tikzpicture}[scale=1.5]
\node[vert] (u) at (0,0) {};
\node[vert] (v) at (1,0) {};
\node[vert] (x) at (2,0) {};
\node[vert] (y) at (3,0) {};
\node[vert] (a) at (1,1) {};
\node[vert] (b) at (2,1) {};
\draw[dashed] (u) -- node[midway,label=below:$e_1$] {}(v);
\draw[dashed] (x) -- node[midway,label=below:$e_2$] {}(y);
\draw[dashed] (a) -- node[midway,label=above:$e'$] {}(b);
\draw[blue] (x) -- (a);
\draw[blue] (y) -- (b);
\draw[red] (u) -- (a);
\draw[red] (v) -- (b);
\end{tikzpicture}
\caption{}
\end{subfigure}
\caption{Two configurations bounded by Lemma~\ref{lem:quasi-random}(\ref{lem:quasi-random-x-lem}).}
\label{fig:x-lem}
\end{figure}
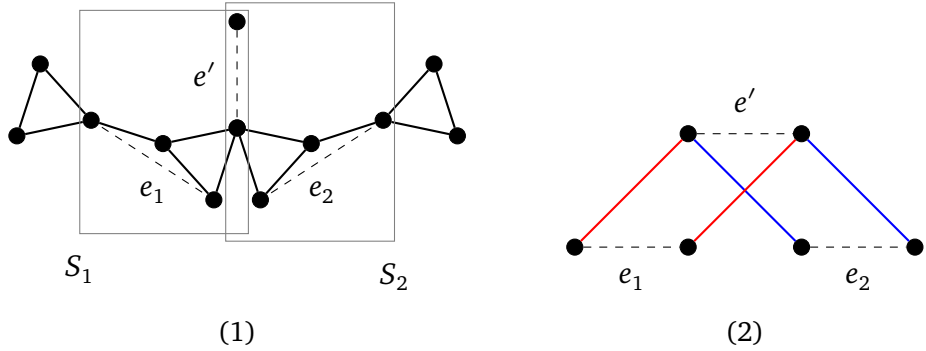

We consider cases based on what role $e'$ plays with $e_1$ and with $e_2$. We have $e' \in \mc{S}_{4, 1}(e_1) \cup \mc{S}_{5, 2}(e_1)$. Figure \ref{fig:ab-types}(3) shows the case $e' \in \mc{S}_{4, 1}(e_1)$, and \ref{fig:ab-types}(7)-(16) shows all possible cases $e' \in \mc{S}_{5, 2}(e_1)$ (where $e'$ is one of the dashed edges and $e_1=uv$). Likewise, $e'$ also plays one of these roles with respect to $e_2$. Thus, the cases we consider are all possible combinations of two roles (one for $e_1$, one for $e_2$) $e'$ could play. The role $e'$ plays with each of  $e_1, e_2$ implies that we certain colored edges which must come from some set of gadgets. For each case, the family of these sets of gadgets (which together would cause $e'$ to play the prescribed role with $e_1$ and $e_2$) will be one of our hypergraphs $\mc{X}$ when we apply Lemma \ref{lem:X}. 

We start with a special case which is a bit different from the rest.  Consider the case shown in Figure~\ref{fig:x-lem}(1), which is when $e'$ plays the role in Figure \ref{fig:ab-types}(14) with $e_1$ and also with $e_2$ and there is a single vertex shared by both gadgets and $e'$. We bound the number of pairs of gadgets $\G_1, \G_2$ contained in our matching $\mc{M}$. Let $\mc{X}$ be the hypergraph where the edges are these pairs $\G_1, \G_2$ (so here $j=2$). We verify that $\mc{X}$ is $(d_A^2, 1/d_A)$-spread. To determine a pair $\G_1, \G_2$ we must choose a total of 7 vertices and 6 colors, so
 \[
 \Delta_0(\mc{X}) = |\mc{X}| = O(n^{13}) \le d_A^2.
 \]
Now if we fix say $\G_1$ and count the number of possibilities for $\G_2$ we have to choose 3 vertices and 3 colors, so
\[
 \Delta_1(\mc{X}) =  O(n^{6}) \le d_A.
 \]
 thus $\mc{X}$ is $(d_A^2, 1/d_A)$-spread and by Lemma \ref{lem:X} we conclude that $\mc{M}$ contains at most $d_A^{\eps_A/12}$ pairs from $\mc{X}$. Note that in this case $e_1, e_2, e' \in E''$ ($e_1$ and $e_2$ are being used as sharable edges and $e'$ must belong to the same set). Since each vertex has sharable degree $O(n^{1-\d})$, the total number of choices for $e'$ is at most $O(n^{1-\d}d_A^{\eps_A/12})=O(n^{1-\d/2})$ by \eqref{eqn:ed}. There is similar case where $e_1$ and $e_2$ share a vertex.

 Note that in Figure \ref{fig:ab-types}(3) and (7)--(16), only in (14) do we see a dashed edge that has an endpoint adjacent to no colored edges. This is what is special about the first case we considered: even after we bounded the number of sets of gadgets, we had a ``free choice'' of the uncolored edge $e'$. In all of the future cases $e'$ will have both endpoints at vertices that are used by one of our gadgets, i.e.\ $e'$ is determined up to $O(1)$ choices by the set of gadgets. We will check two such cases by hand and leave the rest for the reader. We claim that for each such $j$-uniform hypergraph $\mc{X}$ we have that $\mc{X}$ is $(n^{7j}, 1/d_A)$-spread. By  Lemma \ref{lem:X} this implies $\mc{M}$ contains at most $n^{7j}d_A^{-j+\eps/12}=O(n^{1-\d/2}$ edges of $\mc{X}$. Thus, Lemma \ref{lem:quasi-random}(\ref{lem:quasi-random-x-lem}) follows from checking our claim that each $\mc{X}$ is $(n^{7j}, 1/d_A)$-spread. We will check an easy case and a more complicated case in detail and leave the rest to the reader. 
 
Let us check the case where $e_1, e_2$ are disjoint and they both play the role in Figure \ref{fig:ab-types}(3) (see Figure \ref{fig:x-lem}(2)). Say $\G_1, \G_2$ are each responsible for one of the red edges, and $\G_3, \G_4$ are responsible for the blue edges. We let $\mc{X}$ be the set of all such $4$-tuples $\{\G_1, \G_2, \G_3, \G_4\}$. We verify that $\mc{X}$ is $(n^{28}, 1/d_A)$-spread. We 
 have (justification follows)
\[
\Delta_0(\mc{X}) = O(n^{28}),\quad \Delta_1(\mc{X}) = O(n^{20}),\quad \Delta_2(\mc{X}) = O(n^{13}),\quad \Delta_3(\mc{X}) = O(n^6).
\]
Indeed, to justify the $\Delta_0(\mc{X})$ bound we can determine an entire 4-tuple $\{\G_1, \G_2, \G_3, \G_4\}$ by choosing a total of 18 vertices and 10 colors. For $\Delta_1(\mc{X})$ we fix one of our gadgets (say $\G_1$) and note that to determine $\G_2, \G_3, \G_4$ we choose 13 vertices and 7 colors. For $\Delta_2(\mc{X})$ we fix two gadgets and then determine the other two by choosing either 8 vertices and 5 colors or 9 vertices and 4 colors (depending on which two gadgets we fixed). For $\Delta_3(\mc{X})$ we fix say $\G_1, \G_2, \G_3$ and determine $\G_4$ by choosing 4 vertices and 2 colors. Thus $\mc{X}$ is $(n^{28}, 1/d_A)$-spread and we are done with this case.

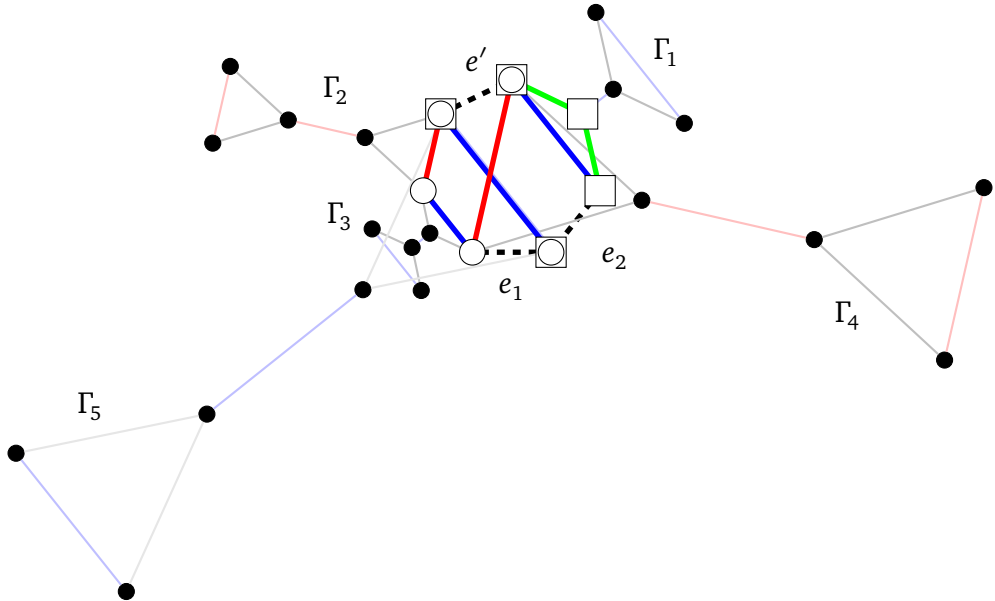
\begin{figure}[ht]
\centering
\begin{tikzpicture}[scale=1.2]
\foreach \i in {0,1,2,3,4,5,6}
	\node[vert] (\i) at ({cos(360 * \i/7 - 2*360/7-90/7)}, {sin(360*\i/7- 2*360/7-90/7}) {};
	
\draw[dashed, , line width=2pt] (0) --node[midway,label=below:$e_1$]{} (1);
\draw[dashed, , line width=2pt] (1) --node[midway,label=below right:$e_2$]{} (2);
\draw[dashed, , line width=2pt] (4) --node[midway,label=above:$e'$]{} (5);

\begin{scope}[shift={(.95,.76)},]
\begin{scope}[rotate=360/7*4+90/7, xscale=0.434756528381, yscale=1.56366296451]
	\DrawGamma{blue!25}{black!25}{green!25}{}
\end{scope}
\end{scope}

\begin{scope}[shift={(-2.04,.46)},]
\begin{scope}[rotate=360/7*7-90/7, xscale=0.867767478, yscale=0.867767478]
	\DrawGamma{red!25}{black!25}{black!25}{$\Gamma_2$}
\end{scope}
\end{scope}

\begin{scope}[shift={(-1.0,-.77)},]
\begin{scope}[rotate=360/7*8-90/7, xscale=0.25, yscale=0.867767478]
	\DrawGamma{blue!25}{black!25}{black!25}{}
\end{scope}
\end{scope}

\begin{scope}[shift={(2.385,-.545)},]
\begin{scope}[rotate=360/7*7-90/7, xscale=1.9499120493, yscale=1.9499120493]
	\DrawGamma{red!25}{black!25}{black!25}{}
\end{scope}
\end{scope}

\begin{scope}[shift={(-2.5,-2.0)},]
\begin{scope}[rotate=360/7*1-90/7, xscale=2.2, yscale=1.9499120493]
	\DrawGamma{blue!25}{black!10}{black!10}{}
\end{scope}
\end{scope}

\foreach \i/\j in {4/3,3/2} {
	\draw[green, line width=2pt] (\i) -- (\j);
}

\foreach \i/\j in {0/6,1/5,2/4} {
	\draw[blue, line width=2pt] (\i) -- (\j);
}

\foreach \i/\j in {5/6,0/4} {
	\draw[red, line width=2pt] (\i) -- (\j);
}

\foreach \i in {1,2,3,4,5} {
	\node[fixed] (\i) at (\i) {};
}
\foreach \i in {4,5,6,0,1} {
	\node[fixed2] (\i) at (\i) {};
}

\node (label1) at (1.7,1.3) [] {$\Gamma_1$};

\node (label3) at (-1.9,-.5) [] {$\Gamma_3$};
\node (label4) at (3.7,-1.6) [] {$\Gamma_4$};
\node (label5) at (-4.65,-2.6) [] {$\Gamma_5$};

\end{tikzpicture}
\caption{${\circ}$ vertices are in $S_1$ and $\square$ vertices are in $S_2$.}
\label{fig:x-lem-3}
\end{figure}

Now we consider a more complicated case shown in Figure~\ref{fig:x-lem-3}. In this case $e_1$ and $e_2$ share a vertex. $e'$ plays the role in Figure \ref{fig:ab-types}(15) with $e_1$, so there is a set $S_1$ of 5 vertices including both $e_1$ and $e'$. $e'$ plays the role in Figure \ref{fig:ab-types}(9) with $e_1$, so there is a set $S_2$ of 5 vertices including both $e_2$ and $e'$. Furthermore we chose to consider a case where $S_1 \cap S_2$ contains an edge (in blue) that contributes to the color repetition within $S_1$ and also within $S_2$. This case will be an application of Lemma \ref{lem:X} with $j=5$. Let $\mc{X}$ be the set of all $5$-tuples $\{\G_1, \G_2, \G_3, \G_4, \G_5\}$ as shown in Figure~\ref{fig:x-lem-3}. Then we have (justification follows)
\[
\Delta_0(\mc{X}) = O(n^{34}),\quad \Delta_1(\mc{X}) = O(n^{26}),\quad \Delta_2(\mc{X}) = O(n^{19}),\quad \Delta_3(\mc{X}) = O(n^{12}),\quad \Delta_4(\mc{X}) = O(n^{6}).
\]
To justify $\Delta_0(\mc{X})$ we can determine a 5-tuple in $\mc{X}$ by choosing a total of 23 vertices and 11 colors, i.e.\ 34 choices for pieces of information where we have $O(n)$ possibilities for each choice. For $\Delta_1(\mc{X})$ note that fixing any one gadget fixes at least 5 vertices not in $e_1 \cup e_2$ and fixes 3 colors (i.e.\ 8 of our 34 choices), so to complete a 5-tuple of $\mc{X}$ we make at most 26 choices. For $\Delta_2(\mc{X})$ note that fixing 2 of our gadgets fixes at least 15 choices for colors and vertices, where an example of the worst case is when we fix $\G_1, \G_4$. For $\Delta_3(\mc{X})$ note that fixing 3 gadgets fixes at least 22 choices for colors and vertices, and the worst case is when we fix $\G_1, \G_4, \G_5$. Finally for $\Delta_4(\mc{X})$ note that if we fix 4 gadgets (a worst case being say $\G_1, \G_2, \G_3, \G_4$) we need to choose at most 4 vertices and 2 colors for the last gadget. Putting together the above bounds, we see that $\mc{X}$ is $(n^{35}, 1/d_A)$-spread as claimed.

\end{document}